\documentclass[final,1p,times]{elsarticle}

\usepackage{amssymb}
\usepackage{amsmath}
\usepackage[utf8]{inputenc}
\usepackage{amsmath}
\usepackage{amsfonts}
\usepackage{color,soul}
\usepackage{placeins}
\usepackage{makecell}
\usepackage{amssymb}
\usepackage{graphicx}
\usepackage{hyperref}
\begin{document}
\newtheorem{thm}{Theorem}[section]
\newtheorem{lem}{Lemma}[section]
\newtheorem{rem}{Remark}[section]
\newtheorem{prop}{Proposition}[section]
\newtheorem{cor}{Corollary}[section]
\newtheorem{example}{Example}[section]
\newdefinition{defn}{Definition}[section]
\newproof{proof}{Proof}
\renewcommand{\theequation}{\thesection.\arabic{equation}}

\begin{frontmatter}



\title{Probability Modelled Averaged Spectrally Optimal Dual Frame and Dual Pair for Erasure}

\author {Shankhadeep Mondal}
\ead{shankhadeep16@iisertvm.ac.in}

\address{School of Mathematics, Indian Institute of Science Education and Research Thiruvananthapuram, Maruthamala P.O, Vithura, Thiruvananthapuram-695551.}

\begin{abstract}
Finding the optimal dual frame and optimal dual pair for signal reconstruction, which can minimize the reconstruction error when erasure occurs during data transmission, is a deep rooted problem from the perspective of frame theory. In this paper, we consider a new measurement for the error operator by taking the average of spectral radius and operator norm with probabilistic erasure. In this measurement, optimal dual frames are called Probabilistic Averaged Spectrally Optimal Dual frames, PASOD-frames in short and optimal dual pair is called PASOD-pair. The properties of the set of PASOD-frames for a pre-selected frame, has been studied. We prove that the set of all PASOD-frames is convex, closed and compact. We also show that the image of a PASOD-frame and PASOD-pair under any unitary operator is also a PASOD-frame and PASOD-pair. We provide several equivalent conditions for the canonical dual to be the unique PASOD-frame for a given frame $F.$ Moreover, we prove non-uniqueness of PASOD-frame under certain condition of the given frame. We also go on to characterize the set of all PASOD-pairs and give several equivalent conditions for a dual pair to become POD, PSOD and PASOD-pair. 
\end{abstract}


\end{frontmatter}

\section{Introduction}
Frames are the generalization of basis of a Hilbert space. In 1952, frames were first introduced by Duffin and Schaeffer\cite{duff}. Initially frames were introduced to deal with some non-harmonic Fourier series problems. Frames have their own useful properties, like frames behave as basis, they spans the whole Hilbert space but their representations may not be unique. In particular, a basis is a frame but a frame may not be a basis. Because of its more flexibility features, frames find applications in various fields such as quantization, noise reduction,image and signal processing, network processing, capacity of transmission channel, coding theory etc.\cite{cand,bodm1,dana,bodm2,holm,jins,alba}.\\
In the data transmission process, frames are used due to their redundancy features. Error in reconstruction of a signal occurs when some part of the data is missing due to transmission. So, redundancy can help recover the signal. The optimal dual problem deals with minimizing the maximal errors under erasures. In a signal communication system, the received data (i.e. the coefficient of received signal vector) may be lost during the transmission process. As a result, the researcher only receives a part of the encoded data. This is known as the erasure problem. Researchers deal with this problem in engineering applications and quite often the erasure event can be characterized by the structure of the communication system. Enhancing robustness to coefficient erasure has been receiving increased attention \cite{benn,cido}. A lot of effort has been put into reducing the reconstruction error caused by the erasure or attempting perfect reconstruction by using the redundancy property of frames\cite{jerr,casa,casa2,jins,sali,holm,deep}. \\
In the transmission process, data erasure arises from buffer overflows at the routers. Bad global conditions such as conjunction can cause capacity of transmission channel error. In these cases, erasures of different elements usually  generate some probabilistic irregularities\cite{leng, miao}. Probability of a bad channel failure is usually larger than the probability of a good channel failure,\\
The erasure problem was first introduced by Paulson and Holmes in the paper "Optimal Frames for Erasure"\cite{holm}. Minimizing the reconstruction error has two approaches: One is finding the optimal Parseval frame \cite{holm} and the other way is to find the optimal dual frame for a given frame that can minimize the error\cite{jerr}. There are various kinds of measurements of the error operator to characterize the optimal dual frame. In\cite{jins}, Leng and Han characterize optimal dual frames for erasure using the operator norm. They gave several equivalent conditions for the optimality of the canonical dual. Pahlivan et.al\cite{sali} have used spectral radius of the error operator to study optimal dual frames.  Practically, it is important to use different measurements of the error operator to obtain optimal dual frames, as it simplifies the computations depending upon the application. In \cite{deep}, Deepshikha and Aniruddha consider the average of operator norm and numerical radius as a measurement of error operator.\\
In this paper, we consider average of spectral radius and operator norm with probabilistic measure as a measurement of error operator. In this measurement, an optimal dual frame is called Probabilistic Averaged Spectrally Optimal Dual frame, in short PASOD-frame. Then we consider the more general case by considering all such dual pairs and characterized PASOD-pairs in the Hilbert space $\mathcal{H}.$ In section 3, we characterize the PASOD-frame for a given  frame $F.$ In general, the canonical dual may not be the optimal dual for a given frame $F.$ We also produce several equivalent conditions for the optimality of the canonical dual of a given frame $F$ for 1-erasure. We also proved that if a dual frame is the PASOD-frame of $F$, then under any unitary operator $U,$ the dual is again the PASOD-frame of $UF,$ for 1-erasure. For a Parseval frame, we establish a relationsip between Probabilistic Spectrally Optimal Dual(PSOD),  Probabilistic Optimal Dual(POD) and  Probabilistic Averaged Spectrally Optimal Dual(PASOD) for 1-erasure. We have also derived some conditions under which the canonical dual is the unique 1-erasure PASOD of a tight frame $F.$ In section 4, we consider the collection of dual pairs and characterize PASOD-pair. We give several equivalent conditions for which a tight frame and it's canonical dual pair achieve optimality and also establish the relationship among POD-pair,PSOD-pair and PASOD-pair.

\section{Preliminaries on Erasures for Probability Model }

Let $\mathcal{H}$ denote an finite dimensional(real or complex) Hilbert space. Throughout this manuscript, we consider $\mathcal{H}$ to be an $n$-dimensional Hilbert space and $N$ be the number of frame elements of $F.$  A finite sequence of elements $F= \{f_i\}_{i=1}^N$ in $\mathcal{H}$ is called a \textit{frame} for $\mathcal{H}$ if there exist constants $A,B >0$ such that
 $$\displaystyle{A \left\| f \right\|^2\leq \sum_{i=1}^N\big|\langle f,f_i\rangle\big|^2 \leq B\left\| f \right\|^2},   \forall f\in \mathcal{H}.$$

 The constants $A$ and $B$ are called frame bounds. They are not unique. The  \textit{optimal lower frame bound} is the supremum over all lower frame bounds and the \textit{optimal upper frame bound} is the infimum over all upper frame bounds.  If $A=B,$ i.e.  $\displaystyle{\sum_{i=1}^N\big| \langle f,f_i \rangle \big|^2 = A\left\| f \right\|^2}$ for all $f\in \mathcal{H},$  then $\{f_i\}_{i=1}^N$ is called a \textit{tight frame.}
     If $A=B=1,$ then $\{f_i\}_{i=1}^N$  is called a \textit{Parseval frame}.
Every finite sequence  $\{f_i\}_{i=1}^N$  in $ \mathcal{H}$ is a frame for the Hilbert space W:= span $\{f_i\}_{i=1}^N$.
 Let $ \mathcal{H}$ be a Hilbert space equipped with a frame $F = \{f_i\}_{i=1}^N.$ Then the linear map $\Theta_F:  \mathcal{H} \to  {\mathbb{C}^N} $ defined by $$\Theta_F(f)= \{\langle f,f_i \rangle\}_{i=1}^N $$ is called  \textit{analysis operator}.\\
      The adjoint operator   $ \Theta_{F}^*: \mathbb{C}^N \to  \mathcal{H} $  defined by
        $$\Theta_{F}^*\left(\{c_i\}_{i=1}^N\right) =  \sum_{i=1}^N {c_i f_i} $$ is called \textit{synthesis operator or preframe operator}.
    The \textit{frame operator} $S_F : \mathcal{H} \to  \mathcal{H} $  defined as
        $$S_{F}f= \Theta_{F}^{*} \Theta_{F} f = \sum_{i=1}^N \langle f,f_i \rangle f_i,$$
        which is a positive, self adjoint, invertible operator on $ \mathcal{H}.$ For any $F \in \mathcal{H},$ we have the reconstruction formula :
        $$f= \sum_{i=1}^N \langle f,S_{F}^{-1}f_i\rangle f_i\ .$$
\noindent
  A frame $G=\{g_i\}_{i=1}^N$ in $ \mathcal{H}$ is called a \textit{dual frame} of $F= \{f_i\}_{i=1}^N$ if every element $f \in   \mathcal{H}$ can be written as $$ f = \sum_{i=1}^N\langle f,f_i \rangle g_i  = \sum_{i=1}^N \langle f,g_i \rangle f_i .$$

\noindent
It is known that $\{S_{F}^{-1}f_i\}_{i=1}^{N}$ is  a frame and is called the canonical or standard dual frame.   There exist infinitely many dual frames (\cite{ole}) $G$ of $F$ in $ \mathcal{H}$ and
every dual frame $G=\{g_i\}_{i=1}^N$ of F is of the form $G=\{S_{F}^{-1}f_i +u_i\}_{i=1}^N,$ where the sequence $\{u_i\}_{i=1}^N$ satisfies $$\sum_{i=1}^N \langle f,f_i \rangle u_i = \sum_{i=1}^N \langle f,u_i \rangle f_i = 0 ,\;\;   \forall f\in \mathcal{H}.$$

\noindent
 For any dual frame $G$ of $F,$  we have
 $\displaystyle{\sum_{i=1}^N\langle g_i,f_i \rangle} = tr(T_F T_{G}^*)=tr (T_{G}^* T_F)= tr(I)=n.$  In particular if $F$ is a Parseval frame, then $\displaystyle{\sum_{i=1}^N \left\|f_i\right\|^2 =n }$.\\
 
 \noindent 
  Let $F = \{f_i\}_{i=1}^N $ be a frame for $ \mathcal{H}$ and $G = \{g_i\}_{i=1}^N$ be a dual of $F.$ Then $(F,G)$ is called a $(N,n)$ dual pair for $ \mathcal{H}.$\\

\noindent
Let $P = \{p_i \}_{i=1}^N$ be the probability sequence, where $p_i$ be the probability of the $i'th$ erasure for $i=1,2,...,N.$ Then $\{p_i \}_{i=1}^N$ must satisfy
 \begin{equation} \label{eqn2.1}
     \displaystyle{\sum_{i=1}^N p_i = 1}, \;\;\; 0 \leq p_i \leq 1 ,\;\; i=1,2,...,N. 
 \end{equation}
The weight number $q_i$ is defined as follows :\\
 \begin{equation} \label{eqn2point2}
 q_i = \displaystyle{{\frac{\sum_{j=1}^N p_j}{\sum_{j=1}^N p_j - p_i}\cdot \frac{N-1}{n}}} ; \;\;\; for\;\; i= 1,2,...,N.\;\;\;\;
\end{equation}

\begin{prop}
Let $ \mathcal{H}$ be an Hilbert space of dimension n and $N \geq n.$  Let $\{q_i\}_{i=1}^N$ be a weight number sequence is as in  (\ref{eqn2point2}) for a frame $F.$ Then $\{q_i\}_{i=1}^N$ satisfy the following properties :
\begin{enumerate}
    \item [{\em (i)}] $\displaystyle{q_i \geq 1 ,\;\; \forall 1 \leq i \leq N },$
    \item [{\em (ii)}] $\displaystyle{\sum_{j=1}^N \frac{1}{q_i}} = n ,$
    \item [{\em (iii)}] If the  number $p_i$ increase then the number $q_i$ also increase.
\end{enumerate}
\end{prop}  \hfill{$\square$}

\vskip 1em
\noindent
 During data transmission if error occurs in \textit{`m'} positions then the error operator is defined by
 \begin{align} \label{eqn2point3}
   E_{\Lambda,(F,G)}f:= \Theta_{G}^*D_{P}\Theta_{F} f=\sum_{i\in\Lambda}  q_i \langle f,f_i \rangle g_i,   
 \end{align}
where $\Lambda $ is the set of indices corresponding to the erased coefficients, $D_p$ is an $N\times N $ diagonal matrix with  diagonal elements $d_{ii}= q_i$ for $i\in \Lambda$ and 0 otherwise.\\

\noindent
Let $F = \{f_i\}_{i=1}^N $ be a frame and $G=\{g_i\}_{i=1}^N$ be a dual frame of $F.$ Let  $\{p_i \}_{i=1}^N$ be a probability sequence given by  (\ref{eqn2.1}) and  $\{q_i\}_{i=1}^N$   be the weight number sequence is given by  (\ref{eqn2point2}).  For $1 \leq m \leq N ,$ let us define $\mathcal{O}_{P}^{(m)}(F,G) := max \bigg\{ \| E_{\Lambda,(F,G)} \| : |\Lambda|=m \bigg\} ,$ where $\| E_{\Lambda,(F,G)} \|$ is the operator norm of $E_{\Lambda,(F,G)}.$ So,  $\mathcal{O}_{P}^{(m)}(F,G)$ is the maximum error for $m$-erasure. A dual frame $G$ of $F$ is called an 1-erasure probabilistic optimal dual  (in short POD) if 
$$\mathcal{O}_P^{(1)}(F,G) = min \left\{ \mathcal{O}_P^{(1)}(F,G')  : G' \;\text{is a dual of F} \right\} .$$

\noindent
For $1 < m \leq N,$ a dual frame $G$ of $F$ is called $m$-erasure POD if $G$ is an $(m-1)$-erasure POD of $F$ and 
$$\mathcal{O}_P^{(m)}(F,G)  = min \left\{ \mathcal{O}_P^{(m)}(F,G')  : G' \;\text{is a dual of F} \right\} .$$
The set of all POD-frames of $F$ for $m$-erasure is defined as $POD_P^{(m)}(F).$  \\~\\

\noindent
For a frame  $F = \{f_i\}_{i=1}^N $ and a dual $G=\{g_i\}_{i=1}^N$ of $F,$ with weight number sequence $\{q_i\}_{i=1}^N$  is given by  (\ref{eqn2point2}), let us define $r_P^{(m)}(F,G) = max \left\{\; \rho( E_{\Lambda,(F,G)}) : |\Lambda|=m \right\} ,$ where $\rho( E_{\Lambda,(F,G)})$ is the spectral radius of $E_{\Lambda,(F,G)}.$ A dual frame $G$ of $F$ is called 1-erasure probabilistic spectrally optimal dual(in short PSOD) of $F$ if 
$$ r_P^{(1)}(F,G) =   min \left\{ r_P^{(1)}(F,G') : G' \;\text{is a dual of } F \right\} .$$

For $1 < m \leq N,$ a dual frame $G$ of $F$ is called $m$-erasure PSOD-frame if $G$ is an $(m-1)$-erasure PSOD of $F$ and 
$$ r_P^{(1)}(F,G) = min \left\{  r_P^{(1)}(F,G') : G' \;\text{is a dual of} F \right\} .$$
 The set of all PSOD-frame of $F$ for $m$-erasure is defined as $PSOD_{P}^{(m)}(F).$
 
\section{Probability modelled Averaged Spectrally Optimal dual frames}
\noindent
Now we will define the concept of Averaged Spectrally Optimal Dual frames for $m$-erasures, $1 \leq m \leq N.$ Let $F = \{f_i\}_{i=1}^N $ be a frame for $\mathcal{H}$ and  $G=\{g_i\}_{i=1}^N$ be a dual of $F$. Let $\{q_i\}_{i=1}^N$ be a weight number sequence as in (\ref{eqn2point2}). The error operator $E_{\Lambda,(F,G)}$ is defined as in equation (\ref{eqn2point3}). The maximum averaged spectral error for a given frame $F$ and its dual $G$ for $m$-erasures is defined as 
 \begin{align} \label{eqn3point2}
 \mathcal{A}_{P}^{(m)}(F,G) = \textit{max} \left\{ \frac{\left\|  E_{\Lambda,(F,G)}  \right\| + \rho{( E_{\Lambda,(F,G)})}}{2}: D_{p} \in \mathcal{D}^{P}_{m}  \right\},
 \end{align}
   
where $\mathcal{D}_{p}^m $ is the set of all diagonal matrices with $`m'$ nonzero entries($q_i$ in $i'th$ position) and  zeroes in $N-m$ entries on the main diagonal.  \\
Now we define 
 \begin{align*}
 \mathcal{A}_{P}^{(m)}(F) :=  inf \left\{ \mathcal{A}_{P}^{(m)}(F,G) : \text{$G$ is a dual of $F$} \right\}.
 \end{align*}
 A dual $G$ of $F$ is called Probabilistic Averaged Spectrally Optimal Dual frame (in short PASOD-frame) of $F$ for 1-erasure if  $\mathcal{A}_{P}^{(1)}(F,G) = \mathcal{A}_{P}^{(m)}(F) .$\\
 In general, a dual $G$ of $F$ is called PASOD-frame  of $F$ for $m$-erasure if it is $(m-1)$-erasure PASOD-frame and  $\mathcal{A}_{P}^{(m)}(F,G) = \mathcal{A}_{F}^{(1)}.$\\
 Let us define 
 $$ \Delta_{F}^{(m)} = \left\{G  \in  \Delta_{F}^{(m-1)} : \mathcal{A}_{P}^{(m)}(F)=\mathcal{A}_{P}^{(m)}(F,G) \right\}  $$
 as a set of all $m$-erasure PASOD-frame of $F.$\\
 
 \begin{prop}
 Let $F = \{f_i\}_{i=1}^N $ be a frame for $\mathcal{H}$ and $\{q_i\}_{i=1}^N$ be a weight number sequence given by  (\ref{eqn2point2}). Let $G=\{g_i\}_{i=1}^N$ be a dual frame of $F.$ Then
 $$ \mathcal{A}_{P}^{(1)}(F,G) = \textit{max}_{i=1}^N \;\; \frac{{q_i\bigg\{| \langle f_i,g_i \rangle | + \|f_i\|\; \|g_i\|\bigg\}}}{2}  .$$
 \end{prop}
 
 \begin{proof}
 For 1-erasure, if the error occurs in the $i'th$ position, then the error operator is $E_{\Lambda,(F,G)}(f) = q_i \langle f,f_i \rangle g_i.$ Therefore $\rho(E_{\Lambda,(F,G)}) = |q_i \langle f_i,g_i \rangle |$ and $\| E_{\Lambda,(F,G)}\| = q_i\|f_i\|\; \|g_i\|.$\\\\
 Hence, $\mathcal{A}_{P}^{(1)}(F,G) = max_{i=1}^N  \frac{\left\|  E_{\Lambda,(F,G)}  \right\| + \rho{\left( E_{\Lambda,(F,G)}\right)}}{2} = max_{i=1}^N \dfrac{{q_i \bigg\{| \langle f_i,g_i \rangle | + \|f_i\|\; \|g_i\|\bigg\}}}{2}.$ \hfill{$\square$}
 \end{proof}  
 
\noindent
Now we will verify some topological properties like convexity, closedness, compactness of $\Delta_{F}^{(1)}.$ In \cite{jerr}, authors have proved that the set of all optimal dual frames is a convex, closed and bounded set for any $m$-erasure. We extend this result for probabilistic averaged spectrally optimal dual frame for erasures.

\begin{thm}
 Let $F = \{f_i\}_{i=1}^N $ be a frame for $\mathcal{H}$ and $\{q_i\}_{i=1}^N$ be a weight number sequence given by  (\ref{eqn2point2}) . Then the set $\Delta_{F}^{(1)} $ is a convex set.
\end{thm}
 \begin{proof}
 Let  $ G= \{g_i\}_{i=1}^N,G'= \{g'_i\}_{i=1}^N \in \Delta_{F}^{(1)} .$ \\
 Therefore $\mathcal{A}_{P}^{(m)}(F,G) = \mathcal{A}_{P}^{(m)}(F,G') = \mathcal{A}_{P}^{(m)}(F)$. Take $\delta \in [0,1]$ be arbitrary.\\
 Set \; $G'' = \delta G + (1-\delta) G' = \left\{ \delta g_i +(1- \delta) g'_{i} \right\}_{i=1}^N,$ be a convex combination of $G$ and $G'.$\\
 It is easy to check that $G''$ is a dual frame of $F.$\\
 Also, $E_{\Lambda,(F,G'')}(f) = \delta E_{\Lambda,(F,G)}(f) + (1-\delta) E_{\Lambda,(F,G')}(f). $\\
 Therefore ,
 \begin{align*}
     \mathcal{A}_{P}^{(1)}(F,G) &= \textit{max}_{i=1}^N \;\; \frac{q_i\bigg\{{| \langle f_i,  \delta g_i +(1- \delta) g'_{i} \rangle | + \|f_i\|\; \| \delta g_i +(1- \delta) g'_{i}\|\bigg\}}}{2} \\ &\leq \textit{max}_{i=1}^N \;\; \frac{{q_i\;\delta |\langle f_i,  g_i \rangle | +q_i(1- \delta) |\langle f_i,g'_{i} \rangle | + q_i\;\delta \|f_i\|\; \| g_i\| + q_i(1- \delta) \|f_i\|\; \|g'_{i}\|}}{2}\\ &= \textit{max}_{i=1}^N \;\; \left\{\delta\; \frac{q_i\bigg\{{| \langle f_i, g_i  \rangle | + \|f_i\|\; \| g_i \|}\bigg\}}{2} + (1-\delta)\; \frac{q_i\bigg\{{| \langle f_i, g'_i  \rangle | + \|f_i\|\; \| g'_i \|}\bigg\}}{2} \right\} \\&\leq  \textit{max}_{i=1}^N \;\; \left\{ \delta\; \frac{q_i\bigg\{{| \langle f_i, g_i  \rangle | + \|f_i\|\; \| g_i \|}\bigg\}}{2} \right\} +  \textit{max}_{i=1}^N \;\; \left\{(1-\delta)\; \frac{q_i\bigg\{{| \langle f_i, g'_i  \rangle | + \|f_i\|\; \| g'_i \|}\bigg\}}{2}\right\} \\&= \delta \mathcal{A}_{P}^{(1)}(F,G) + (1-\delta)\mathcal{A}_{P}^{(1)}(F,G') \\&= \delta \mathcal{A}_{P}^{(1)}(F) + (1-\delta)\mathcal{A}_{P}^{(1)}(F) \\&= \mathcal{A}_{P}^{(1)}(F)
 \end{align*}
  This implies $G'' \in \Delta_{F}^{(1)} $ and hence $\Delta_{F}^{(1)} $ is a convex set. \\
 \end{proof}     \hfill{$\square$}
\begin{rem}
In fact, for a frame  $F = \{f_i\}_{i=1}^N $ with weight number sequence $\{q_i\}_{i=1}^N,$ the set $\Delta_{F}^{(m)} $ is a convex set for $1 \leq m \leq N.$
\end{rem}
\begin{prop}
 Let $F = \{f_i\}_{i=1}^N $ be a frame for $\mathcal{H}$ and $\{q_i\}_{i=1}^N$ be a weight number sequence given by  (\ref{eqn2point2}) . Then the set $\Delta_{F}^{(m)} $ is a closed set for $1\leq m \leq N$.
\end{prop}
\begin{thm}
 Let $F = \{f_i\}_{i=1}^N $ be a frame for $\mathcal{H}$ and $\{q_i\}_{i=1}^N$ be a weight number sequence given by  (\ref{eqn2point2}) . Then the set $\Delta_{F}^{(1)} $ is a compact subset in the set of all dual frames of $F$ in $\mathcal{H}$.

\end{thm}

\begin{proof}
Let $\zeta := \left\{\text{collection of all dual frames of $F$ in } \mathcal{H}  \right\}.$\\
This can be easily seen that $\zeta$ is a complete subspace of  $\mathcal{H}.$\\
Let us define a map 
\begin{eqnarray*}
& \sharp : \zeta \rightarrow \mathbb{R}^{+} \cup \{0\}\;\;\; \text{by}\\&
   \sharp (\{S_{F}^{-1}f_i +h_i\}_{i=1}^N) =  \mathcal{A}_{P}^{(1)}(F, \{S_{F}^{-1}f_i +h_i\}_{i=1}^N ) =  \textit{max}_{i=1}^N \left\{ \frac{{q_i | \langle f_i,S_{F}^{-1}f_i +h_i \rangle | + q_i \|f_i\|\; \|S_{F}^{-1}f_i +h_i\|}}{2} \right\}
\end{eqnarray*}
Therefore, $\sharp$ is a continuous function.\\
Let $\gamma = \sharp\left(\{ S_{F}^{-1}f_i \}_{i=1}^N \right) = \textit{max}_{i=1}^N \left\{ \dfrac{{q_i | \langle f_i,S_{F}^{-1}f_i\rangle | + q_i \|f_i\|\; \|S_{F}^{-1}f_i \|}}{2} \right\} < \infty$\\
Now, consider the interval $[0,\gamma].\; \sharp^{-1}\left([0,\gamma]\right)$  is a compact subset of $\mathcal{H}^{(N)}.$ Therefore it attains maximum and minimum value i.e there exists a $G \in \sharp^{-1}\left([0,\gamma]\right)$ which attain the minimum value say $\beta.$ Now, $ \Delta_{F}^{(1)} = \sharp^{-1}(\beta) \subset \sharp^{-1}\left([0,\gamma]\right),$ be a non-empty set. Hence, $\Delta_{F}^{(1)}$ is a compact set.   \hfill{$\square$}
\end{proof}

\noindent
The image of a dual pair under any unitary operator again forms a dual pair. The following theorem proves that the image of an 1-erasure PASOD-frame of $F$ is also an 1-erasure PASOD-frame of the unitary image of $F.$  
\begin{thm}
  Let $F = \{f_i\}_{i=1}^N $ be a frame for $\mathcal{H}$ and  $\{q_i\}_{i=1}^N$ be a weight number sequence given by  (\ref{eqn2point2}) . $U$ be a unitary operator in $\mathcal{H}.$ Then for any dual $G \in \Delta_{F}^{(1)}$ if and only if $UG \in \Delta_{UF}^{(1)}.$
\end{thm}
 \begin{proof}
 Since $U$ is a unitary operator in $\mathcal{H},$ then $UF = \{Uf_i\}_{i=1}^N$ is a frame for $\mathcal{H}$ with dual frame $UG =  \{Ug_i\}_{i=1}^N.$\\
 Let $G \in \Delta_{F}^{(1)}$ and $G'= \{g'_i\}_{i=1}^N$ be a dual of $UF$ in $\mathcal{H}.$ Then $U^{*}G'$ is a dual of $F$ in $\mathcal{H}.$\\
 Therefore,
 \begin{align*}
     \mathcal{A}_{P}^{(1)}(UF,UG) =& \textit{max}_{i=1}^N \left\{ \dfrac{{q_i| \langle Uf_i,Ug_i\rangle | + q_i\|Uf_i\|\; \|Ug_i \|}}{2} \right\} \\&= \textit{max}_{i=1}^N \left\{ q_i\dfrac{{| \langle f_i,g_i\rangle | + \|f_i\|\; \|g_i \|}}{2} \right\} \\ &\leq \mathcal{A}_{P}^{(1)}(F,U^{*}G') \\ &= \textit{max}_{i=1}^N \left\{ q_i \dfrac{{| \langle f_i,U^*g'_i\rangle | + \|f_i\|\; \|U^*g'_i \|}}{2} \right\} \\&= \textit{max}_{i=1}^N \left\{ q_i\dfrac{{| \langle Uf_i,g'_i\rangle | + \|Uf_i\|\; \|g'_i \|}}{2} \right\} \\&= \mathcal{A}_{P}^{(1)}(UF,G')
 \end{align*}
 Hence $UG \in \Delta_{UF}^{(1)}.$\\
 Conversely, if $UG \in \Delta_{UF}^{(1)}.$ Let $G''= \{g''_i\}_{i=1}^N$ be a dual of $F$ in $\mathcal{H}.$\\
 Then $UG''= \{Ug''_i\}_{i=1}^N $ be a dual of $UF$ in  $\mathcal{H}.$\\
 Therefore, 
 \begin{align*}
      \mathcal{A}_{P}^{(1)}(UF) &=  \mathcal{A}_{P}^{(1)}(UF,UG) =  \textit{max}_{i=1}^N \left\{ q_i\dfrac{{| \langle Uf_i,Ug_i\rangle | + \|Uf_i\|\; \|Ug_i \|}}{2} \right\} \leq \textit{max}_{i=1}^N \left\{ q_i \dfrac{{| \langle Uf_i,Ug''_i\rangle | + \|Uf_i\|\; \|Ug''_i \|}}{2} \right\}
 \end{align*}
 i.e., $$\textit{max}_{i=1}^N \left\{ q_i\dfrac{{| \langle f_i,g_i\rangle | + \|f_i\|\; \|g_i \|}}{2} \right\} \leq \textit{max}_{i=1}^N \left\{ q_i\dfrac{{| \langle f_i,g''_i\rangle | + \|f_i\|\; \|g''_i \|}}{2} \right\}$$
 Therefore, $\mathcal{A}_{P}^{(1)}(F,G) \leq \mathcal{A}_{P}^{(1)}(F,G''),$ for any dual $G''$ of $F$ in $\mathcal{H}.$ Hence $ G \in \Delta_{F}^{(1)}.$      
 \end{proof}   \hfill{$\square$}  
 
 The following theorem establish a relation between POD and PASOD for 1-erasure for a given tight frame $F.$ Note that this may not true for any general frame.
 
 \begin{thm}\label{thm3point4}
  Let $F = \{f_i\}_{i=1}^N $ be a tight frame for $\mathcal{H}$ and  $\{q_i\}_{i=1}^N$ be a weight number sequence given by  (\ref{eqn2point2}). Then $S_{F}^{-1} F$ is an 1-erasure POD of $F$ if and only if $S_{F}^{-1} F \in \Delta_{F}^{(1)}.$ 
 \end{thm}

\begin{proof}
Let $F$ be a tight frame with tight bound $A.$ \\
First consider the canonical dual $S_{F}^{-1} F = \{\frac{1}{A} f_{i}\}_{i=1}^N \in \Delta_{F}^{(1)}.$ \\
Let $G= \{g_i\}_{i=1}^N$ be any dual of $F.$ Then,
\begin{align*}
   \mathcal{O}_P^{(1)}(F,S_{F}^{-1} F)= max_{i=1}^N q_i\|f_i\| \left\|\tfrac{1}{A} f_{i}\right\| &= max_{i=1}^N \frac{q_i|\langle f_i,\tfrac{1}{A} f_{i} \rangle | + q_i\|f_i\| \left\|\tfrac{1}{A} f_{i}\right\| }{2} \\&= \mathcal{A}_{P}^{(1)}(F,\tfrac{1}{A}F)  \\&\leq \mathcal{A}_{P}^{(1)}(F,G)  \\&=  max_{i=1}^N \frac{q_i|\langle f_i,g_i \rangle | + q_i\|f_i\|\; \| g_{i}\| }{2} \\&\leq  max_{i=1}^N \frac{q_i\|f_i\|\;\|g_i\| + q_i\|f_i\|\; \| g_{i}\| }{2} \\&= \mathcal{O}_P^{(1)}(F,G)
\end{align*}
Hence $S_{F}^{-1} F $ is an  1-erasure probabilistic optimal dual of $F.$\\
\noindent
For the converse part, let $G=\left\{ \tfrac{1}{A} f_{i} + h_i \right\}_{i=1}^N$ be any non-canonical dual of $F$ such that $G \in \Delta_{F}^{(1)}$ and  $S_{F}^{-1} F  \notin  \Delta_{F}^{(1)}.$\\
Thus,
\begin{align} \label{eqn3point3}
  \mathcal{A}_{P}^{(1)}(F,G)&=  max_{i=1}^N \frac{q_i|\langle f_i,\tfrac{1}{A} f_{i} + h_i \rangle | + q_i\|f_i\|\; \left\|\tfrac{1}{A} f_{i} +h_i\right\| }{2} \nonumber\\&=  max_{i=1}^N \frac{q_i\left| \tfrac{1}{A} \|f_{i}\|^2 + \langle f_i, h_i\rangle \right| + q_i\sqrt{\frac{1}{A^2}\|f_i\|^4 + \|f_i\|^2\|h_i\|^2 + \frac{2}{A}\|f_i\|^2 Re( \langle f_i, h_i\rangle)} }{2} \nonumber \\&<  \mathcal{A}_{P}^{(1)}(F,S_{F}^{-1}F) \nonumber\\&= max_{i=1}^N \;\;\tfrac{q_i}{A} \|f_i\|^2. 
\end{align}
Now, let us consider $\Lambda_{1} = \left\{ i :  \frac{q_i}{A}\|f_i\|^2 = max_{i=1}^N  \frac{q_i}{A}\|f_i\|^2 \right\}$  and $\Lambda_{2} = \{1,2,...,N\} \setminus \Lambda_{1}. $\\
From equation (\ref{eqn3point3}), for all $i \in \Lambda_{1}, \text{Re}(\langle f_i, h_i\rangle) < 0.$ Therefore, for all  $i \in \Lambda_{1},$ we can choose $\epsilon_{1}^{i} > 0,$ small enough such that 
\begin{align*}
    q_i\|f_i\|\;\|\tfrac{1}{A}f_i + \epsilon_{1}^{i}h_i \| =q_i\sqrt{\frac{1}{A^2}\|f_i\|^4 + (\epsilon_{1}^{i})^2\|f_i\|^2\|h_i\|^2 + \frac{2}{A}\epsilon_{1}^{i}\|f_i\|^2 Re( \langle f_i, h_i\rangle)} \;\;<\; q_i\frac{1}{A}\|f_i\|^2 = max_{i=1}^N \frac{q_i}{A}\|f_i\|^2 .
\end{align*}
Now, for all $j \in \Lambda_{2},\;\; \frac{q_j}{A}\|f_j\|^2 < max_{i=1}^N  \frac{q_i}{A}\|f_i\|^2 .$ \\
Therefore,  for all $j \in \Lambda_{2},$ there exists  $\epsilon_{2}^{j} > 0,$ small enough such that 
\small
\begin{align*}
   q_j \|f_j\|\;\|\tfrac{1}{A}f_j + \epsilon_{2}^{j}h_j \| =q_j\sqrt{\frac{1}{A^2}\|f_j\|^4 + (\epsilon_{2}^{j})^2\|f_j\|^2\|h_j\|^2 + \frac{2}{A}\epsilon_{2}^{j}\|f_i\|^2 Re( \langle f_j, h_j \rangle)} \;\;<\; q_j\frac{1}{A}\|f_i\|^2 < max_{i=1}^N \frac{q_i}{A}\|f_i\|^2 .
\end{align*}
\normalsize
Take $\epsilon = \displaystyle{min_{i \in \Lambda_{1},j \in \Lambda_{2}}} \;\left\{\epsilon_{1}^{i}, \epsilon_{2}^{j} \right\}.$\\ 
Thence, $max_{i=1}^N  q_i \|f_i\| \; \left\|\tfrac{1}{A}f_i + \epsilon h_i \right\| \;<\;  max_{i=1}^N \frac{q_i}{A}\|f_i\|^2.$\\
As a consequence, for the dual $G_{\epsilon} = \left\{\frac{1}{A}f_i + \epsilon h_i    \right\}_{i=1}^N,\;\;   \mathcal{O}_P^{(1)}(F,G_{\epsilon}) <   \mathcal{O}_P^{(1)}(F,S_{F}^{-1} F)$ holds.\\
This gives a contradiction that  $S_{F}^{-1} F $ is an 1-erasure POD-frame of $F.$ Therefore our assumption that $S_{F}^{-1} F  \notin  \Delta_{F}^{(1)}$ is not true. Hence, $S_{F}^{-1} F  \in  \Delta_{F}^{(1)}.$ 
 \end{proof} \hfill{$\square$}

\noindent 
Now we will give an equivalent condition between 1-erasure PSOD and PASOD for frame $F,$ when $F$ is tight. This may not be true for any general frame.
\begin{thm}
Let $F = \{f_i\}_{i=1}^N $ be a tight frame for $\mathcal{H}$ and  $\{q_i\}_{i=1}^N$ be a weight number sequence given by  (\ref{eqn2point2}). Then the canonical dual $S_{F}^{-1} F$ be an 1-erasure PSOD-frame of $F$ if and only if $S_{F}^{-1} F  \in \Delta_{F}^{(1)}.$ 
\end{thm}
\begin{proof}
Let $F$ be a tight frame with tight bound $A.$ \\
First suppose that, $S_{F}^{-1} F \in \Delta_{F}^{(1)}. $  \\ 
Let $G = \left\{\frac{1}{A}f_i +h_i\right\}_{i=1}^N$ be a non-canonical dual of $F$ such that $r_{P}^{(1)}(F,G) < r_{P}^{(1)}(F, S_{F}^{-1} F).$\\
\noindent
Let us consider $\Lambda_{1} = \left\{ i :  \frac{q_i}{A}\|f_i\|^2 = max_{i=1}^N  \frac{q_i}{A}\|f_i\|^2 \right\}$  and $\Lambda_{2} = \{1,2,...,N\} \setminus \Lambda_{1}$. \\
Thus, $max_{i=1}^N \;\; q_i |\langle f_i,\tfrac{1}{A} f_{i} +h_i \rangle| < max_{i=1}^N \;\; q_i |\langle f_i,\tfrac{1}{A} f_{i} \rangle| .$ This implies $ q_i |\langle f_i,\tfrac{1}{A} f_{i} +h_i \rangle| <  q_i |\langle f_i,\tfrac{1}{A} f_{i} \rangle|, $ for all $ i \in \Lambda_1.$ Therefore, for all $i \in \Lambda_1,\;\; \left| \tfrac{q_i}{A} \|f_i\|^2 + q_i \langle f_i , h_i \rangle  \right| <  \tfrac{q_i}{A} \|f_i\|^2. $ Consequently, $Re(\langle f_i , h_i \rangle) < 0,$\;for all $i \in \Lambda_1.$ Thus for all $i \in \Lambda_1,$ there exists $\epsilon_1 > 0$ \; small enough such that
\begin{align*}
   \frac{q_i |\langle f_i,\tfrac{1}{A} f_{i} + \epsilon_1 h_i\rangle | + q_i \|f_i\|\; \left\|\tfrac{1}{A} f_{i} + \epsilon_1 h_i\right\| }{2} &\leq q_i \|f_i\|\; \left\|\tfrac{1}{A} f_{i} + \epsilon_1 h_i\right\| \\&= \sqrt{\frac{q_i^2}{A^2}\|f_i\|^4 + q_i^2 \epsilon_1^{2}\|f_i\|^2\;\|h_i\|^2 + 2\frac{\epsilon_1}{A}q_i^2 \|f_i\|^2 Re\langle f_i , h_i \rangle} \\&< \frac{q_i}{A}\|f_i\|^2  .
\end{align*}
For all $i \in \Lambda_2, \;\; \frac{q_i}{A}\|f_i\|^2 < max_{i=1}^N \frac{q_i}{A}\|f_i\|^2.$ Therefore, for all $i \in \Lambda_2,$ there exists $\epsilon_2 > 0$ \; small enough such that
\begin{align*}
    \frac{q_i |\langle f_i,\tfrac{1}{A} f_{i} + \epsilon_2 h_i\rangle | + q_i \|f_i\|\; \left\|\tfrac{1}{A} f_{i} + \epsilon_2 h_i\right\| }{2} &\leq q_i \|f_i\|\; \left\|\tfrac{1}{A} f_{i} + \epsilon_2 h_i\right\| \\&= \sqrt{\frac{q_i^2}{A^2}\|f_i\|^4 + q_i^2 \epsilon_2^{2}\|f_i\|^2\;\|h_i\|^2 + 2\frac{\epsilon_2}{A}q_i^2 \|f_i\|^2 re\langle f_i , h_i \rangle} \\& < max_{i=1}^N \;\frac{q_i}{A}\|f_i\|^2 .
\end{align*}
Take $\epsilon = min\{\epsilon_1, \epsilon_2 \}. $\\
Therefore, $max_{i=1}^N \;\;\dfrac{q_i \left|\langle f_i,\tfrac{1}{A} f_{i} + \epsilon_1 h_i\rangle \right| + q_i \|f_i\|\;\left\|\tfrac{1}{A} f_{i} + \epsilon_1 h_i\right\| }{2} < max_{i=1}^N \;\frac{q_i}{A}\|f_i\|^2  ,\;\;$ for all $1 \leq i \leq N.$
 Consequently, for the dual $G_{\epsilon}= \left\{\frac{1}{A}f_i + \epsilon h_i\right\}_{i=1}^N ,\;\; \mathcal{A}_{P}^{(1)}(F,G_{\epsilon}) < \mathcal{A}_{P}^{(1)}(F,S_{F}^{-1}F) .$
This gives a contradiction that $S_{F}^{-1} F  \in \Delta_{F}^{(1)}$. Therefore, $S_{F}^{-1} F$ is an 1-erasure PSOD of $F.$ \\

\noindent
For the converse part,  for any dual  $G = \left\{\frac{1}{A}f_i +h_i\right\}_{i=1}^N$ of $F$ in $\mathcal{H},$
\begin{align*}
   \mathcal{A}_{P}^{(1)}(F,G) &= max_{i=1}^N \frac{q_i \left|\langle f_i,\tfrac{1}{A} f_{i} + h_i\rangle \right| + q_i \|f_i\|\; \left\|\tfrac{1}{A} f_{i} +h_i\right\| }{2} \\& \geq  max_{i=1}^N \;q_i |\langle f_i,\tfrac{1}{A} f_{i} + h_i\rangle | \\&=  r_{P}^{(1)}(F,G) \\&\geq  r_{P}^{(1)}(F,S_{F}^{-1}F) \\&= max_{i=1}^N \frac{q_i}{A} \|f_i\|^2 \\&=  \mathcal{A}_{P}^{(1)}(F,S_{F}^{-1}F) 
\end{align*}
 Hence, $S_{F}^{-1}F \in \Delta_{F}^{(1)}.$
\end{proof}\hfill{$\square$}

\begin{thm} \label{thm3point5}
Let $F = \{f_i\}_{i=1}^N $ be a tight frame for $\mathcal{H}$  and  $\{q_i\}_{i=1}^N$ be a weight number sequence given by  (\ref{eqn2point2}).  Then $S_{F}^{-1} F$ is the unique probabilistic 1-erasure optimal dual of $F$ if and only if $S_{F}^{-1} F $ is the unique 1-erasure PASOD-frame of $F.$ 
\end{thm}
\begin{proof}
The proof is similar to theorem \ref{thm3point4}.
\end{proof}\hfill{$\square$}

\begin{cor}
 Let $F = \{f_i\}_{i=1}^N $ be a parseval frame for $\mathcal{H}$  and  $\{q_i\}_{i=1}^N$ be a weight number sequence given by  (\ref{eqn2point2}). Then following assertions are equivalent;
 \begin{enumerate}
    \item [{\em (i)}] The canonical dual is an 1-erasure POD of $F.$
    \item [{\em (ii)}] The canonical dual is an 1-erasure PSOD of $F.$
    \item [{\em (iii)}] The canonical dual is an 1-erasure PASOD of $F.$
\end{enumerate} 
\end{cor} \hfill{$\square$}

\noindent
Let $F = \{f_i\}_{i=1}^N $ be a frame for $\mathcal{H}$ and  $\{q_i\}_{i=1}^N$ be a weight number sequence given by  (\ref{eqn2point2}). Let, $c= max_{i=1}^N \left\{ q_i\left\| S_{F}^{-1}f_i \right\|\;\|f_i\| : 1 \leq i \leq N \right\}.$ Set $\eta_1 = \left\{ i :  \left\| S_{F}^{-1}f_i \right\|\; \|f_i\| = c   \right\} $ and $\eta_2 = \{1,2,\dots,N\} \setminus \eta_1.$ Set $F_j = \text{span} \left\{f_i : i \in \Lambda_j  \right\},$ for $j=1,2.$\\
\begin{thm}\cite{leng3}\label{thm3point6}
Let $F = \{f_i\}_{i=1}^N $ be a frame for $\mathcal{H}$ and  $\{q_i\}_{i=1}^N$ be a weight number sequence given by  (\ref{eqn2point2}). Then following are equivalent;
\begin{enumerate}
    \item [{\em (i)}] The canonical dual $\left\{ S_{F}^{-1}f_i \right\}_{i=1}^N$ is the unique probabilistic optimal dual for 1-erasure.
     \item [{\em (ii)}] $F_1 \cap F_2 = \{0 \}$ and $\{f_i \}_{i \in \eta_2}$ is linearly independent.
\end{enumerate}
\end{thm} \hfill{$\square$}

\begin{cor}\cite{leng3}\label{cor3point1}
Let $F = \{f_i\}_{i=1}^N $ be a tight frame for $\mathcal{H}$ and  $\{q_i\}_{i=1}^N$ be a weight number sequence given by  (\ref{eqn2point2}). Then the canonical dual is the unique 1-erasure POD of $F$ if and only if $q_i \|f_i\|^2 $ is constant, $1 \leq i \leq N.$
\end{cor}\hfill{$\square$}

\begin{thm}\label{thm3point7}
Let $F = \{f_i\}_{i=1}^N $ be a tight frame for $\mathcal{H}$ and  $\{q_i\}_{i=1}^N$ be a weight number sequence given by  (\ref{eqn2point2}).  Then the canonical dual is the unique 1-erasure PASOD-frame of $F$ if and only if $q_i\|f_i \|^2 = c,$ for all $1 \leq i \leq N$ and for some constant $c.$
\end{thm}

\begin{proof}
Let $F = \{f_i\}_{i=1}^N,$ be a tight frame with tight bound $A.$ \\
If $q_i\|f_i \|^2 = c,$ for all $1 \leq i \leq N$ then by corollary \ref{cor3point1}, the canonical dual of $F$ is the unique POD-frame of $F$ for 1-erasure. \\
For any dual $G$ of $F,$
$$ \mathcal{A}_{P}^{(1)}(F,G)  =  max_{i=1}^N \frac{q_i|\langle f_i, g_i\rangle | + q_i \|f_i\|\; \|g_i\| }{2} \geq   max_{i=1}^N q_i \|f_i\|\; \|g_i\| \geq max_{i=1}^N \frac{q_i}{A}\|f_i\|^2 = \mathcal{A}_{P}^{(1)}(F,S_{F}^{-1}F) .$$
Hence, canonical dual is a PASOD-frame of $F$ for 1-erasure.
Let $G= \left\{\frac{1}{A}f_i + h_i\right\}_{i=1}^N$ be a non-canonical dual of $F$ such that $G \in  \Delta_{F}^{(1)}.$\\
Therefore,
\begin{align*}
    \mathcal{A}_{P}^{(1)}(F,G) &= max_{i=1}^N \frac{q_i|\langle f_i, \frac{1}{A}f_{i} + h_i\rangle | + q_i \|f_i\|\; \left\|\frac{1}{A} f_{i} +h_i\right\| }{2} \\&= max_{i=1}^N \frac{\left| \frac{q_i}{A}\|f_{i}\|^2 + q_i\langle f_i, h_i\rangle \right| + \sqrt{\frac{q_i^2}{A^2} \|f_i\|^4 + q_i^2\|f_i\|^2\|h_i\|^2 + 2\frac{q_i^2}{A} \|f_i\|^2 Re( \langle f_i, h_i\rangle)} }{2} \\&= max_{i=1}^N \frac{\left| \frac{c}{A} + q_i\langle f_i, h_i\rangle \right| + \sqrt{\frac{c^2}{A^2} + cq_i\|h_i\|^2 + 2\frac{cq_i}{A} Re( \langle f_i, h_i\rangle)} }{2} \\&= max_{i=1}^N \frac{q_i}{A}\|f_i\|^2 \\&= \frac{c}{A},
\end{align*}
which implies that $Re( \langle f_i, h_i\rangle) < 0, 1\leq i \leq N.$\\
Accordingly, for all $1\leq i \leq N,$ there exists $\epsilon_i >0$ small enough such that 
\begin{align}
    q_i\| f_i\|\; \left\| \tfrac{1}{A} f_i + \epsilon_{i}h_i\right\|=\sqrt{\tfrac{q_i^2}{A^2}\|f_i\|^4 + q_i^2\epsilon^2_{i}\|f_i\|^2\|h_i\|^2 + 2\tfrac{q_i^2\epsilon_{i}}{A}\|f_i\|^2 Re( \langle f_i, h_i\rangle)} <  \frac{q_i}{A}\| f_i\|^2 = c
\end{align}
Take $\epsilon = min\{\epsilon_i : 1 \leq i \leq N \}.$\\
Therefore, for all $1 \leq i \leq N,$
$$  q_i\| f_i\|\; \left\| \tfrac{1}{A} f_i + \epsilon h_i\right\|=\sqrt{\tfrac{q_i^2}{A^2}\|f_i\|^4 + q_i^2\epsilon^2\|f_i\|^2\|h_i\|^2 + 2\tfrac{q_i^2\epsilon}{A}\|f_i\|^2 Re( \langle f_i, h_i\rangle)} < \frac{ q_i}{A}\| f_i\|^2 = c.$$
Thus, $max_{i=1}^N q_i\| f_i\|\; \| \frac{1}{A}f_i + \epsilon h_i\| < max_{i=1}^N \frac{q_i}{A}\| f_i\|^2 .$\\
So, for the dual $\tilde{G}=\{\frac{1}{A}f_i+\epsilon h_i\}_{i=1}^N,$\;$\mathcal{O}_{F,\tilde{G}}^{(1)} < \mathcal{O}_{F,S_{F}^{-1}F}^{(1)}.$\\
This gives a contradiction that the canonical dual is the unique POD-frame for 1-erasure. Hence, the canonical dual is the unique 1-erasure PASOD-frame of $F.$\\
For the converse part, if the canonical dual is the unique 1-erasure PASOD of $F,$ then by Theorem \ref{thm3point5} the canonical dual is the unique probabilistic 1-erasure optimal dual of $F.$ Then by corollary\ref{cor3point1} $q_i\|f_i\|^2 = c,\;\;1\leq i \leq N.$\\

\end{proof}\hfill{$\square$}\\
Now we are going to give a construction of the collection of 1-erasure PASOD-frame of $F.$ Also we will show that for $N>n,\;\Delta_F^{(1)}$  is not unique.\\~\\
\noindent
Let $F = \{f_i\}_{i=1}^N $ be a frame for $\mathcal{H}$ and $\{q_i\}_{i=1}^N$ be a weight number sequence given by  (\ref{eqn2point2}). Set $l= max_{i=1}^N \left\{ q_i\left\| S_{F}^{-1/2}f_i \right\|^2 + q_i\|f_i\|\;\left\| S_{F}^{-1}f_i\right\| : 1 \leq i \leq N \right\}.$ Set $\Lambda_1 = \left\{ i : q_i \left\| S_{F}^{-1/2}f_i \right\|^2 + q_i\|f_i\|\;\left\| S_{F}^{-1}f_i\right\| = l   \right\} $ and $\Lambda_2 = \{1,2,..N\} \setminus \Lambda_1.$ Set $H_j = \text{span} \left\{f_i : i \in \Lambda_j  \right\},$ for $j=1,2.$\\

\begin{thm} \label{thm3point8}
Let $F = \{f_i\}_{i=1}^N $ be a frame for $\mathcal{H}.$ If $H_1 \cap H_2 = \{0\}$ and $\{f_i\}_{i\in \Lambda_1}$  is linearly independent then $S_{F}^{-1} F \in \Delta_{F}^{(1)}.$ Moreover, if $N > n$ then $S_{F}^{-1} F$ is not the unique 1-erasure PASOD-frame of $F.$
\end{thm}

\begin{proof}
Let $G= \left\{ S_{F}^{-1}f_i + u_i \right\}_{i=1}^N$ be a dual of $F.$\\
Then, $\sum_{i=1}^N \langle f,u_i \rangle f_i =0,$ which implies,  $\sum_{i \in \Lambda_1} \langle f,u_i \rangle f_i =0 = \sum_{i \in \Lambda_2} \langle f,u_i \rangle f_i. $\\
Using the condition $\{f_i\}_{i\in \Lambda_1}$  is linearly independent, we have $\langle f,u_i \rangle  =0,$ for all \;$i \in \Lambda_1$ and for all\; $i \in \mathcal{H}.$ Further implies, $u_i = 0,$ for all $i \in \Lambda _1.$\\
Also, from the fact that $\sum_{i \in \Lambda_2} \langle f,u_i \rangle f_i =0,$ we have  $T_{F^2}^{*}T_{U^2} =0,$ where $F^2 = \{f_i\}_{ i \in \Lambda_2}$ and $U^2 = \{u_i \}_{ i \in \Lambda_2  }.$ This implies $tr\left(T_{F^2}^{*}T_{U^2}\right) = tr\left(T_{U^2} T_{F^2}^{*} \right)= \sum_{i \in \Lambda_2} \langle f_i , u_i \rangle = 0.$\\
Therefore,
\begin{align*}
   \mathcal{A}_P^{(1)}(F,G) &=  max_{i=1}^N \frac{q_i\left|\langle f_i,S_{F}^{-1} f_{i} + u_i \rangle \right| + q_i\|f_i\| \left\|S_{F}^{-1} f_{i} + u_i\right\| }{2} \\& \geq max_{i \in \Lambda_1} q_i\frac{\left|\langle f_i,S_{F}^{-1} f_{i} + u_i \rangle \right| + \|f_i\| \left\|S_{F}^{-1} f_{i} + u_i\right\| }{2} \\&= max_{i \in \Lambda_1} q_i\frac{\left\| S_{F}^{-1/2} f_{i} \right\|^2 + \|f_i\| \left\|S_{F}^{-1} f_{i} \right\| }{2} \\&=l \\&= max_{i =1}^N q_i\frac{\left\| S_{F}^{-1/2} f_{i} \right\|^2 + \|f_i\| \left\|S_{F}^{-1} f_{i} \right\| }{2} \\&= \mathcal{A}_{p}^{(1)}(F,S_{F}^{-1}F )
\end{align*}
Hence, the canonical dual is an 1-erasure PASOD-frame of $F.$\\
Let us consider $N >n.$ Then there exist a dual $G = \{g_i\}_{i=1}^N = \left\{S_{F}^{-1} f_{i} + u_i \right\}_{i=1}^N$ of $F$ such that $u_i =0$ for all $i \in \Lambda_1$ and $u_i \neq 0$ for some $i \in \Lambda_2.$ As $q_i\left\{|\langle f_i, S_{F}^{-1} f_{i} \rangle | + \|f_i \| \left\| S_{F}^{-1} f_{i}  \right\|\right\} < l,$ for all $i \in \Lambda_2,$ then there exists $\epsilon >0 $ small enough such that $q_i\left\{|\langle f_i, S_{F}^{-1} f_{i} + \epsilon u_i \rangle | + \|f_i \| \left\| S_{F}^{-1} f_{i} + \epsilon u_i \right\|\right\} < l,$ for all $i \in \Lambda_2.$\\

Therefore, for the dual $\Tilde{G} = \{ S_{F}^{-1} f_{i} + \epsilon u_i\}_{i=1}^N, $ we have 
\begin{align*}
\mathcal{A}_{P}^{(1)}(F,\Tilde{G}) &=  max_{i=1}^N \frac{q_i\left|\langle f_i,S_{F}^{-1} f_{i} + \epsilon u_i \rangle \right| + q_i\|f_i\| \left\|S_{F}^{-1} f_{i} + \epsilon u_i\right\| }{2} \\& = \frac{l}{2} \\& =  max_{i=1}^N \frac{q_i\left|\langle f_i,S_{F}^{-1} f_{i} \rangle \right| + q_i\|f_i\| \left\|S_{F}^{-1} f_{i}\right\| }{2} \\&= \mathcal{A}_{P}^{(1)}(F,S_{F}^{-1}F)
\end{align*}
Hence, $\Tilde{G} = \{ S_{F}^{-1} f_{i} + \epsilon u_i\}_{i=1}^N \in \Delta_{F}^{(1)}.$
\end{proof}\hfill{$\square$}




\section{Averaged Spectrally Optimal Dual Pair for Erasure }

 Now we will define the concept of Probabilistic Averaged Spectrally Optimal Dual pair for any $m$-erasures, $1 \leq m \leq N.$ Let $F = \{f_i\}_{i=1}^N $ be a frame for  $\mathcal{H}$  and  $G=\{g_i\}_{i=1}^N$ be a dual frame of $F$. $\{q_i\}_{i=1}^N$ be a weight number sequence given by  (\ref{eqn2point2}). The error operator $E_{\Lambda,(F,G)}$ is defined as in equation (\ref{eqn2point3}). The maximum probabilistic averaged spectral error for a dual pair $(F,G)$ for $m$-erasure is same as $ \mathcal{A}_{P}^{(m)}(F,G).$ Now we define \\
 $$ \mathcal{A}_{P}^{(m)} := inf \left\{ \mathcal{A}_{P}^{(m)}(F,G) : (F,G) \text{is a (N,n) dual pair in } \mathcal{H} \right\}.$$

A  dual pair $(F,G)$ is called 1-erasure Probabilistic Averaged Spectrally Optimal Dual pair (in short PASOD-pair) if $\mathcal{A}_{P}^{(1)}(F,G)  = \mathcal{A}_{P}^{(1)}. $\\
In general, a dual pair$(F,G)$ is called $m$-erasure PASOD-pair if it is $(m-1)$-erasure PASOD-pair and  $\mathcal{A}_{P}^{(m)}(F,G)  = \mathcal{A}_{P}^{(m)}. $\\
Let us define
$$ \zeta_{P}^{(m)} := \left\{ (F,G) : (F,G) \in \zeta^{(m-1)} \;\text{and}\;\; \mathcal{A}_{P}^{(m)}(F,G) = \mathcal{A}_{P}^{(m)} \right\}$$
as the set of $m$-erasure PASOD-pair in   $\mathcal{H}.$\\
\noindent
Similarly, we will define Probabilistic Spectrally optimal dual pair( in short PSOD-pair) and Probabilistic Optimal dual pair (POD-pair) for a Hilbert space $\mathcal{H}.$ Let us define
$$ r_P^{(m)} := inf \bigg\{ r_P^{(m)}(F,G) : (F,G) \text{ is a (N,n) dual pair } \bigg\} $$
$$ \mathcal{O}_P^{(m)} := inf \bigg\{ \mathcal{O}_P^{(m)}(F,G) : (F,G) \text{ is a (N,n) dual pair } \bigg\} $$
A dual pair $(F,G)$ is  called $m$-erasure PSOD-pair if it is $(m-1)$ erasure PSOD-pair and $r_P^{(m)}(F,G) = r_P^{(m)}.$ A dual pair $(F,G)$ is  called $m$-erasure POD-pair if it is $(m-1)$ erasure POD-pair and $\mathcal{O}_P^{(m)}(F,G) = \mathcal{O}_P^{(m)}.$ 
\begin{defn}
Let $\{p_i\}_{i=1}^N$ be a probability sequence given by (\ref{eqn2.1}). $\{q_i\}_{i=1}^N$ be the corresponding weight number sequence defined as  (\ref{eqn2point2}). We call a parseval frame $\{f_i\}_{i=1}^N$ is a probability uniform parseval frame if it satisfies 
$$ \|f_i\| = \frac{1}{\sqrt{q_i}} $$  for all $1 \leq i \leq N.$
\end{defn}
Now we will prove the existance of $(N,n)$ probability uniform parseval frame by using the following result.
\begin{thm} \cite{cass2} \label{thm4point1}
Let S be a positive self-adjoint operator on an n-dimensional Hilbert space $\mathcal{H}$ .Let $\lambda_1 \geq \lambda_2 \geq ...\geq \lambda_n \gneq 0$ be the eigenvalues of $S$. Fix $N \geq n$ and the real numbers $a_1 \geq a_2 \geq ... \geq a_N \gneq 0$. Then the following are equivalent;\\
\begin{enumerate}
     \item[{\em (i)}] There is a frame $\{f_i \}_{i=1}^N$ for $H$ with frame operator $S$ and $\|f_i \| = a_i$ , for all $1 \leq i \leq N $.\\
     \item[{\em (ii)}] For every $1 \leq k \leq n$,\\~\\
     $\displaystyle{ \sum_{i=1}^k a_{i}^2 \leq \sum_{i=1}^k\lambda_{i}} $  and  $\displaystyle{ \sum_{i=1}^N a_{i}^2 = \sum_{i=1}^n \lambda_i } $.
\end{enumerate}
   \end{thm}\hfill{$\square$} 
   
\begin{cor}\label{cor4point1}
 Let $\mathcal{H}$ be an $n$-dimensional Hilbert space. Let $\{p_i\}_{i=1}^N$ be a probability sequence given by (\ref{eqn2.1}) and $\{q_i\}_{i=1}^N$ be the corresponding weight number sequence defined as  (\ref{eqn2point2}).  Then there always exists a probability uniform parseval frame $F = \{f_i\}_{i=1}^N.$
\end{cor} 
\begin{proof}
As in Theorem \ref{thm4point1}, we let $S = I,$ be the $n \times n$ identity matrix. The eigenvalues of $S$ are $\lambda_i =1, \;\; 1\leq i \leq n.$ Now we permute the sequence $\{q_i\}_{i=1}^N$ such that $q'_1 \leq q'_2 \leq \dots \leq q'_k. $\\
Let $$ a'_i = \frac{1}{\sqrt{q'_i}}, \;\;\; 1 \leq i \leq N. $$
Then for $1 \leq k \leq n,$\;\;\; $\displaystyle{\sum_{i=1}^k a_{i}'^2  \leq k = \sum_{i=1}^k \lambda_i }$ and $\displaystyle{\sum_{i=1}^N a_{i}'^2 = \sum_{i=1}^N \frac{1}{q'_i} = n =  \sum_{i=1}^n \lambda_i.} $\\
Therefore, by Theorem \ref{thm4point1}, there exists a frame $\{f_i \}_{i=1}^N$ for $H$  with $I_{n \times n}$ as frame operator and $\|f_i \| = \frac{1}{\sqrt{q'_i}}.$ Thus $\{f_i \}_{i=1}^N$ is a parseval frame. Permuting $\{f_i \}_{i=1}^N$ in a proper way, we get a parseval frame $F = \{f'_i \}_{i=1}^N$ with $\|f'_i \| =  \frac{1}{\sqrt{q_i}}, \;\; 1 \leq i \leq N.$
\end{proof} \hfill{$\square$}   
   
\begin{prop}
The value of $ \mathcal{A}_{P}^{(1)}$ is 1.
\end{prop}
\begin{proof}
For any $(N,n)$ dual pair $(F,G)$, it is easy to see that $$ \mathcal{A}_{P}^{(1)}(F,G) = \textit{max}_{i=1}^N \;\; \frac{q_i\bigg\{| \langle f_i,g_i \rangle | + \|f_i\|\; \|g_i\|\bigg\}}{2}  \geq max_{i=1}^N\;\; q_i|\langle f_i,g_i \rangle| .$$
Claim: Optimality occurs when $|\langle f_i,g_i \rangle| = \|f_i\|\;\|g_i\| = \frac{1}{q_i},\;\;1 \leq i \leq N.$\\
When $|\langle f_i,g_i \rangle| = \|f_i\|\;\|g_i\| = \frac{1}{q_i},\;\;\text{for all}\; 1 \leq i \leq N,$  then $\mathcal{A}_{P}^{(1)}(F,G) =1.$\\
Let $(\tilde{F}, \tilde{G})$ be any dual pair in $\mathcal{H},$ where $\tilde{F} = \{ \tilde{f_i} \}_{i=1}^N \;\;\text{and}\;\; \tilde{G} = \{ \tilde{f_i} \}_{i=1}^N.$\\
If $\mathcal{A}_{P}^{(1)}(\tilde{F},\tilde{G}) < 1,$ then $max_{i=1}^N q_i |\langle \tilde{f}_i,\tilde{g}_i \rangle| \leq \textit{max}_{i=1}^N \;\; q_i\frac{{| \langle \tilde{f}_i, \tilde{g}_i \rangle | + \|\tilde{f}_i\|\; \|\tilde{g}_i\|}}{2} < 1.$ This implies $q_i |\langle \tilde{f}_i,\tilde{g}_i \rangle| < 1 , \;\; 1\leq i \leq N.$ Therefore, $ n= \sum_{i=1}^N \langle \tilde{f}_i,\tilde{g}_i \rangle \leq \sum_{i=1}^N |\langle \tilde{f}_i,\tilde{g}_i \rangle| < \sum_{i=1}^N \frac{1}{q_i} =n. $ This is not possible. Therefore $\mathcal{A}_{P}^{(1)}(\tilde{F},\tilde{G}) \geq 1 ,$ for any dual pair $(\tilde{F},\tilde{G})$ with weight number sequence $\{q_i\}_{i=1}^N.$ Hence $ \mathcal{A}_{P}^{(1)} \geq 1.$\\
If $F$ be a probability uniform parsevel frame and $G$ be its canonical dual, then $\mathcal{A}_{P}^{(1)}(F,G) =1.$ Therefore the value of $ \mathcal{A}_{P}^{(1)}$ is 1.\\

\end{proof} \hfill{$\square$}
\begin{rem}
For a Hilbert space $\mathcal{H}$ of dimension n and for $N \geq n,$
$$ \zeta_{P}^{(1)} = \left\{ (F,G) : |\langle f_i,g_i \rangle| = \|f_i\|\;\|g_i\| = \frac{1}{q_i},\;\;1 \leq i \leq N \right\}.$$
\end{rem}
The next proposition proves that any 1-erasure PASOD pair will remain an 1-erasure PASOD pair under any unitary map.
\begin{prop}
Let $(F,G)$ be a $(N,n)$ dual pair in $\mathcal{H}.$ Let $U $ be a unitary opeartor in $\mathcal{H}.$ $\{q_i\}_{i=1}^N$ be a weight number sequence given by  (\ref{eqn2point2}). Then $(F,G) \in \mathcal{A}_{P}^{(1)}, $ if and only if $(UF,UG) \in \mathcal{A}_{P}^{(1)}.$
\end{prop}
\begin{proof}
It can be easily seen that 
\begin{align*}
\mathcal{A}_P^{(1)}(UF,UG) &= max_{i=1}^N \frac{q_i\bigg\{| \langle Uf_i,Ug_i \rangle | + \|Uf_i\|\; \|Ug_i\| \bigg\}}{2} \\&= max_{i=1}^N \frac{q_i \bigg\{| \langle f_i,g_i \rangle | + \|f_i\|\; \|g_i\| \bigg\}}{2} \\&=\mathcal{A}_{P}(F,G). \end{align*}
Hence the result follows.
\end{proof}\hfill{$\square$}

\noindent
Now, we are going to give a necessary and sufficient condition for  1-erasure PSOD-pair and 1-erasure POD-pair.
\begin{prop} \label{prop4point3}
An (N,n) dual pair $(F,G)$ in  $\mathcal{H}$ with  weight number sequence $\{q_i\}_{i=1}^N,$  given by  (\ref{eqn2point2}), is an 1-erasure PSOD-pair if and only if \;$| \langle f_i, g_i \rangle | = \langle f_i, g_i \rangle = \frac{1}{q_i},\;\text{for all} \; 1\leq i \leq N.$
\end{prop}
\begin{proof}
For any $(N,n)$ dual pair $(F,G)$ in  $\mathcal{H}$, $r_{P}^{(1)}(F,G) = max_{i=1}^N \; q_i|\langle f_i,g_i \rangle |.$\\
Claim: Optimality occurs when $| \langle f_i, g_i \rangle | = \langle f_i, g_i \rangle = \frac{1}{q_i},\;\text{for all} \; 1\leq i \leq N.$\\
For a dual pair $(F,G),$ when  $| \langle f_i, g_i \rangle | = \langle f_i, g_i \rangle = \frac{1}{q_i},\; 1\leq i \leq N$ holds,then  $r_{P}^{(1)}(F,G) =1.$\\
Let $(\tilde{F}, \tilde{G})$ be any dual pair in $\mathcal{H},$ where $\tilde{F} = \{ \tilde{f_i} \}_{i=1}^N \;\;\text{and}\;\; \tilde{G} = \{ \tilde{f_i} \}_{i=1}^N.$\\
If $r_{P}^{(1)}(\tilde{F},\tilde{G}) < 1,$ then $max_{i=1}^N \; q_i|\langle \tilde{f}_i,\tilde{g}_i \rangle | < 1.$ Therefore, $\displaystyle{ n = \sum_{i=1}^N \langle \tilde{f}_i,\tilde{g}_i \rangle \leq \sum_{i=1}^N | \langle \tilde{f}_i,\tilde{g}_i \rangle| < \sum_{i=1}^N \frac{1}{q_i} =n}. $ This gives a contradiction. As a consequence, we can say, $r_{P}^{(1)}(\tilde{F},\tilde{G}) \geq 1 ,$ for any dual pair $(\tilde{F},\tilde{G})$ with weight number sequence $\{q_i\}_{i=1}^N.$ Hence $r_{P}^{(1)} \geq 1.$\\
If $F$ be a probability uniform parsevel frame and $G$ be its canonical dual, then $r_{P}^{(1)}(F,G) =1.$ Therefore, the value of $ {r}_{P}^{(1)}$ is 1.\\
Accordingly, if $(F,G)$ be an 1-erasure probabilistic spectrally optimal dual pair if and only if $max_{i=1}^N q_i |\langle f_i, g_i \rangle | = 1.$ From the facts $\sum_{i=1}^N \langle f_i, g_i \rangle = n$ and $ \sum_{i=1}^N \frac{1}{q_i} = n,$  $(F,G)$ is an 1-erasure PSOD-pair if and only if $\langle f_i, g_i \rangle = |\langle f_i, g_i \rangle| = \frac{1}{q_i},\;\;\;1 \leq i \leq N.$
\end{proof} \hfill{$\square$}

\begin{prop} \label{prop4point4}
An (N,n) dual pair $(F,G)$ in  $\mathcal{H}$ with  weight number sequence $\{q_i\}_{i=1}^N,$  given by  (\ref{eqn2point2}), is an 1-erasure  POD-pair if and only if $ \langle f_i, g_i \rangle = \|f_i \|\; \|g_i\| = \frac{1}{q_i},$ for all $1\leq i \leq N.$
\end{prop}
\begin{proof}
For any $(N,n)$ dual pair $(F,G)$  with  weight number sequence $\{q_i\}_{i=1}^N$ in  $\mathcal{H}$, $\mathcal{O}_{F,G}^{(1)} =\\ max_{i=1}^N q_i\;\|f_i \|\; \|g_i\| .$\\
Claim: Optimality occurs when $\|f_i\|\;\|g_i\| = \frac{1}{q_i},\;\text{for all} \; 1\leq i \leq N.$\\
When $\|f_i\|\;\|g_i\| = \frac{1}{q_i},\;\;1 \leq i \leq N$ holds, then $\mathcal{O}_{P}^{(1)}(F,G) =1.$\\
For any dual pair $(\tilde{F}, \tilde{G}),$ if $\mathcal{O}_{P}^{(1)}(\tilde{F}, \tilde{G}) < 1,$ then $max_{i=1}^N \; q_i\| \tilde{f}_i\|\;\|\tilde{g}_i \| < 1.$ Therefore, $ n = \sum_{i=1}^N \langle \tilde{f}_i,\tilde{g}_i \rangle \leq \sum_{i=1}^N \|\tilde{f}_i\|\;\|\tilde{g}_i \| < \sum_{i=1}^N \frac{1}{q_i} =n. $ This gives a contradiction. Therefore, $\mathcal{O}_{P}^{(1)}(\tilde{F},\tilde{G}) \geq 1 ,$ for any dual pair $(\tilde{F},\tilde{G})$ with weight number sequence $\{q_i\}_{i=1}^N.$ Hence $\mathcal{O}_{P}^{(1)} \geq 1.$\\
If $F$ be a probability uniform parsevel frame and $G$ be its canonical dual, then $\mathcal{O}_{P}^{(1)}(F,G) =1.$ Therefore the value of $ \mathcal{O}_{P}^{(1)}$ is 1.\\
Consequently, if $(F,G)$ be an 1-erasure POD-pair if and only if $max_{i=1}^N q_i \| f_i\|\;\| g_i \| = 1.$ From the facts $\sum_{i=1}^N \langle f_i, g_i \rangle = n$ and $ \sum_{i=1}^N \frac{1}{q_i} = n,$ we have  $(F,G)$ be an 1-erasure POD-pair if and only if $\langle f_i, g_i \rangle =\|f_i\|\;\|g_i\| = \frac{1}{q_i},\;\;\;1 \leq i \leq N.$
\end{proof} \hfill{$\square$}

The following theorem gives a relation between PSOD-pair and PASOD-pair for 1-erasure.

\begin{thm} \label{thm4point1}
Let $F= \{f_i\}_{i=1}^N$ be a tight frame for  $\mathcal{H}$  with  weight number sequence $\{q_i\}_{i=1}^N$  given by  (\ref{eqn2point2}). Then, $(F, S_{F}^{-1}F )$ be an 1-erasure PSOD-pair if and only if  $(F, S_{F}^{-1}F ) \in \zeta_{P}^{(1)}.$ 
\end{thm}
\begin{proof}
Let $F= \{f_i\}_{i=1}^N$ be a tight frame with tight bound $A.$ \\
If $(F, S_{F}^{-1}F )$ be an 1-erasure PSOD-pair in  $\mathcal{H},$ then by proposition (\ref{prop4point3}), $\langle f_i , S_{F}^{-1}f_i \rangle = \frac{1}{A}\|f_i\|^2 = \frac{1}{q_i} .$\\
Therefore, $\mathcal{A}_{P}^{(1)}(F,S_{F}^{-1}F)=  max_{i=1}^N \frac{q_i\bigg\{| \langle f_i,\frac{1}{A}f_i \rangle | + \|f_i\|\; \|\frac{1}{A}f_i \| \bigg\}}{2} =  max_{i=1}^N \frac{q_i}{A}\|f_i\|^2 = 1 = \mathcal{A}_{P}^{(1)}.$\\~\\
Conversely, if $ (F, S_{F}^{-1}F ) \in \zeta_{P}^{(1)},$ then $\mathcal{A}_{P}^{(1)}(F,S_{F}^{-1}F)=  max_{i=1}^N \frac{q_i}{A}\|f_i\|^2 =1.$ If for any $j,\;\;1 \leq j \leq N,$ $\frac{q_j}{A}\|f_j\|^2 < 1,$ then $n = \sum_{i=1}^N \langle f_i, \frac{1}{A} f_i \rangle = \sum_{i=1}^N \frac{1}{A}\|f_i\|^2 < \sum_{i=1}^N \frac{1}{q_i} = n. $ This is not possible. Therefore, $\frac{1}{A}\|f_i\|^2 = \frac{1}{q_i},\;\; 1 \leq i \leq N$ and hence $r_{P}^{(1)} (F,S_{F}^{-1}F) = max_{i=1}^N \frac{q_i}{A}\|f_j\|^2  = 1. $ So,  $(F, S_{F}^{-1}F )$ is an 1-erasure PSOD-pair in  $\mathcal{H}$.
\end{proof}  \hfill{$\square$}

\begin{thm}\label{thm4point2}
Let $F= \{f_i\}_{i=1}^N$ be a tight frame for  $\mathcal{H}$ with  weight number sequence $\{q_i\}_{i=1}^N,$  given by  (\ref{eqn2point2}). Then, $(F, S_{F}^{-1}F )$ be an 1-erasure POD-pair if and only if  $(F, S_{F}^{-1}F ) \in \zeta_{P}^{(1)}.$ 
\end{thm}
\begin{proof}
The proof is similar to the Theorem (\ref{thm4point1}).
\end{proof} \hfill{$\square$}

The following theorem gives the equivalent relation of uniqueness between 1-erasure POD-pair and PASOD-pair. Note that, here we use  uniqueness in the sense of unique dual of the tight frame $F.$
\begin{thm} \label{thm4point3}
Let $F= \{f_i\}_{i=1}^N$ be a tight frame for  $\mathcal{H}$ with  weight number sequence $\{q_i\}_{i=1}^N$  given by  (\ref{eqn2point2}). Then, $(F, S_{F}^{-1}F )$ be the unique  POD-pair if and only if  $(F, S_{F}^{-1}F )$  is the unique  PASOD-pair in $\mathcal{H}$ for  1-erasure.
\end{thm}
\begin{proof}
Let $F= \{f_i\}_{i=1}^N$ be a tight frame with tight bound $A.$ \\
If $(F, S_{F}^{-1}F )$ is the unique 1-erasure PASOD-pair in  $\mathcal{H},$ then $\mathcal{A}_{P}^{(1)}(F,S_{F}^{-1}F)=  1 = \max_{i=1}^N \tfrac{q_i}{A}\|f_i\|^2.$\\
Using the fact $\sum_{i=1}^N \langle f_i , S_{F}^{-1}f_i \rangle = \sum_{i=1}^N \frac{1}{A}\|f_i\|^2 = n $ and the similar argument as in Proposition(\ref{thm4point1}), we can show that $\frac{1}{A}\|f_i\|^2 = \frac{1}{q_i},$ for all $1 \leq i \leq N.$\\
Now, $\mathcal{O}_{P}^{(1)}(F, S_{F}^{-1}F) = \max_{i=1}^N q_i\|f_i \|\;\|S_{F}^{-1}f_{i}\| = \max_{i=1}^N \frac{q_i}{A}\|f_i\|^2 =1 .$ Therefore, $(F, S_{F}^{-1}F )$ is an 1-erasure POD-pair in  $\mathcal{H}$.\\
Let $G= \left\{\frac{1}{A}f_i + h_i\right\}_{i=1}^N$ be a dual of $F$ such that $(F,G)$ is an 1-erasure POD-pair in $\mathcal{H}$. Therefore, $1 = \mathcal{O}_{P}^{(1)}(F,G) =  max_{i=1}^N q_i\|f_i\|\;\|\frac{1}{A}f_i + h_i\| .$\\
If for any j, $1\leq j \leq N,$ \;$q_i\|f_i\|\;\|\frac{1}{A}f_i + h_i\| < 1 ,$ then 
$$n= \sum_{i=1}^N \frac{1}{q_i} > \sum_{i=1}^N \|f_i\|\;\left\|\tfrac{1}{A}f_i + h_i\right\| \geq \sum_{i=1}^N \left| \langle f_j, \frac{1}{A}f_i + h_i \rangle \right| \geq \left| \sum_{i=1}^N \langle f_i, \tfrac{1}{A}f_i + h_i \rangle \right| =n. $$
This is not possible. Therefore $q_i\|f_i\|\;\left\|\tfrac{1}{A}f_i + h_i\right\| = 1,$ for all $1 \leq i \leq N.$\\
Consequently, $q_i\|f_i\|\;\left\|\frac{1}{A}f_i + h_i\right\| = \frac{q_i}{A}\|f_i\|^2,$ for all $1 \leq i \leq N.$ This implies, $\frac{1}{A^2}\|f_i\|^2 = \left\|\tfrac{1}{A}f_i + h_i\right\|^2= \frac{1}{A^2}\|f_i\|^2  + \|h_i \|^2 + \frac{2}{A} Re (\langle f_i, h_i\rangle ),$\;\;for all $1 \leq i \leq N.$ i.e\; $\|h_i \|^2 + \frac{2}{A} Re (\langle f_i, h_i\rangle ) =0,$\;\;for all $1 \leq i \leq N.$ Taking sum over $1 \leq i \leq N,$ we have $\sum_{i=1}^N \|h_i\|^2 + \frac{2}{A} Re \left( \sum_{j=1}^N \langle f_i, h_i \rangle \right) =0.$ Using the condition $\sum_{i=1}^N \langle f_i, h_i \rangle =0 \left(Re \sum_{i=1}^N \langle f_i, h_i \rangle =0 \right),$ we have $\sum_{i=1}^N \|h_i\|^2 = 0. $ This implies $h_i =0,$ for all $1 \leq i \leq N.$ Hence, $(F, S_{F}^{-1}F )$  is the unique 1-erasure POD-pair in   $\mathcal{H}.$\\

\noindent
Conversely, let $(F,S_{F}^{-1}F)$ be the unique probabilistic 1-erasure optimal dual pair in  $\mathcal{H}.$\\
Therefore, by proposition(\ref{prop4point4}), $q_i\|f_i\|\;\| S_{F}^{-1}f_i\| = \frac{q_i}{A}\|f_i\|^2 = 1 ,$ for all $1 \leq i \leq N.$ Therefore $\mathcal{A}_{P}^{(1)}(F,S_{F}^{-1}F) = max_{i=1}^N \frac{q_i}{A}\|f_i\|^2 = 1.$\\
Let $G= \left\{\frac{1}{A}f_i + h_i\right\}_{i=1}^N$ be a dual of $F$ such that $(F,G)$ be an 1-erasure PASOD-pair in $\mathcal{H}$. Therefore, $ 1 =\frac{q_i}{A}\|f_i\|^2 =  max_{i=1}^N \;q_i\frac{{| \langle f_i,\frac{1}{A}f_i + h_i\rangle | + \|f_i\|\; \|\frac{1}{A}f_i +h_i \|}}{2} \geq max_{i=1}^N \; q_i\left| \langle f_i,\frac{1}{A}f_i + h_i\rangle \right| .$ This implies $1 \geq q_i\left| \langle f_i,\frac{1}{A}f_i + h_i\rangle \right|,$ for all $1 \leq i \leq N.$  Taking sum over $1 \leq i \leq N,$ we have $n = \sum_{i=1}^N \frac{1}{q_i}\geq \sum_{i=1}^N \left| \langle f_i,\frac{1}{A}f_i + h_i\rangle \right| \geq \left| \sum_{i=1}^N  \langle f_i,\frac{1}{A}f_i + h_i\rangle \right| =n. $ Therefore, $n= \sum_{i=1}^N \left| \langle f_i,\frac{1}{A}f_i + h_i\rangle \right|.$
If for any j, $1\leq j \leq N,$ \;$q_j\left| \langle f_j,\frac{1}{A}f_j + h_j\rangle \right| < 1 ,$ then $n \leq \sum_{i=1}^N \left| \langle f_i,\frac{1}{A}f_i + h_i\rangle \right| < \sum_{i=1}^N \frac{1}{q_i} =n, $ which is not possible. Therefore, $q_i \left| \langle f_i,\frac{1}{A}f_i + h_i\rangle \right| =1,$ for all $1 \leq i \leq N.$ This implies $q_i \left| \langle f_i,\frac{1}{A}f_i + h_i\rangle \right| = \frac{q_i}{A}\|f_i\|^2.$  Accordingly, $\left| \frac{1}{A} \|f_i\|^2 + \langle f_i,h_i \rangle \right| = \frac{1}{A}\|f_i\|^2,$ for all $1 \leq i \leq N.$ Therefore, $Re(\langle f_i,h_i \rangle) \leq 0,$ \;for all $1 \leq i \leq N.$ Also from the fact $\sum_{i=1}^N \langle f_i,h_i \rangle =0 \left( \sum_{i=1}^N Re(\langle f_i,h_i \rangle )=0 \right),$ we can conclude that $Re(\langle f_i,h_i \rangle )=0,$ for all $1 \leq i \leq N.$\\

As $(F,G)$ is an 1-erasure PASOD-pair in $\mathcal{H},$ then
\begin{align*}
     1&= max_{i=1}^N\;\; q_i\frac{\frac{1}{A} \|f_i\|^2 + \|f_i\|\;\| \tfrac{1}{A}f_i + h_i \|}{2}\\&=  max_{i=1}^N \frac{1}{2} \left\{ 1 + q_i\|f_i\|\;\|\tfrac{1}{A}f_i + h_i \| \right\} \\&=  max_{i=1}^N \frac{1}{2} \left\{ 1 + \sqrt{\frac{q_i^2}{A^2}\|f_i\|^4 + q_i^2 \|f_i\|^2\;\|h_i\|^2 +  q_i^2 \|f_i\|^2 \frac{2}{A} Re(\langle f_i,h_i \rangle)} \right\} \\&=   max_{i=1}^N \frac{1}{2} \bigg\{ 1 + \sqrt{1 +  q_i^2 \|f_i\|^2\;\|h_i\|^2} \bigg\}.
\end{align*}
As a consequence,\;\; $h_i =0,\;\;1 \leq i \leq N.$  Hence,  $(F,S_{F}^{-1}F )$ is the unique 1-erasure PASOD-pair in  $\mathcal{H}.$ 
\end{proof} \hfill{$\square$}

\begin{cor}
 Let $F = \{f_i\}_{i=1}^N $ be a tight frame for $\mathcal{H}$  and  $\{q_i\}_{i=1}^N$ be a weight number sequence given by  (\ref{eqn2point2}). Then following are equivalent;
 \begin{enumerate}
    \item [{\em (i)}] $(F,S_{F}^{-1}F)$ is an 1-erasure POD-pair.
    \item [{\em (ii)}] $(F,S_{F}^{-1}F)$ is an 1-erasure PSOD-pair.
    \item [{\em (iii)}] $(F,S_{F}^{-1}F)$ is an 1-erasure PASOD-pair.
\end{enumerate} 
\end{cor} \hfill{$\square$}

\begin{thm}\cite{jins} \label{thm4point5}
Let $F = \{f_i\}_{i=1}^N$ be a (N, n)-frame for $\mathcal{H}.$ If $F$ is a uniform tight frame for H, then the canonical dual is the unique optimal dual
frame of $F$ for m-erasures.
\end{thm}\hfill{$\square$}

\begin{prop}
Let $F= \{f_i\}_{i=1}^N$ be a tight frame for  $\mathcal{H}$ with  weight number sequence $\{q_i\}_{i=1}^N$  given by  (\ref{eqn2point2}). Then $(F, S_{F}^{-1}F )$ dual pair is the unique 1-erasure PASOD-pair if and only if $q_i\|f_i\|^2=c,$ for some constant $c$ and for all $1 \leq i \leq N.$
\end{prop}
\begin{proof}
Let $F= \{f_i\}_{i=1}^N$ be a tight frame with tight bound $A.$\\
If  $q_i\|f_i\|^2=c, \; 1 \leq i \leq N,$ then we have $n = \sum_{i=1}^N \langle f_i, \frac{1}{A}f_i \rangle = \sum_{i=1}^N \frac{1}{A} \|f_i\|^2 =  \sum_{i=1}^N \;\frac{c}{Aq_i}= \frac{cn}{A}.$ This implies $c=A.$ Therefore,  $\|f_i \|\;\|\frac{1}{A} f_i\| = \langle f_i,\frac{1}{A} f_i \rangle = \frac{1}{A} \|f_i\|^2 = \frac{1}{q_i} ,\;1 \leq i \leq N.$ Therefore by Proposition \ref{prop4point4}, $(F, S_{F}^{-1}F )$ is an 1-erasure optimal dual pair in $\mathcal{H}.$\\
Let $G= \left\{\frac{1}{A}f_i + h_i\right\}_{i=1}^N$ be a dual of $F$ such that $(F,G)$ is an 1-erasure POD-pair in $\mathcal{H}.$ Therefore, 
$$ \mathcal{O}_{P}^{(1)}(F, S_{F}^{-1}F) = max_{i=1}^N \frac{q_i}{A}\|f_i \|^2 = 1 = \mathcal{O}_{P}^{(1)}(F,G) = max_{i=1}^N q_i\|f_i \|\; \| \tfrac{1}{A}f_i + h_i\|. $$
Using similar argument as in Theorem \ref{thm4point3}, we can show that $h_i =0,\;\; 1\leq i\leq N.$ Therefore  $(F, S_{F}^{-1}F )$ be the unique 1-erasure POD-pair in $\mathcal{H}.$ Then by Theorem \ref{thm4point3}, $(F, S_{F}^{-1}F )$  is the unique 1-erasure PASOD-pair in $\mathcal{H}.$

Conversely, if  $(F, S_{F}^{-1}F )$ is the unique 1-erasure PASOD-pair in $\mathcal{H},$ then $\mathcal{A}_{P}^{(1)}(F,S_{F}^{-1}F) = max_{i=1}^N q_i\frac{{| \langle f_i,\frac{1}{A}f_i \rangle | + \|f_i\|\; \|\frac{1}{A}f_i \| }}{2} =  max_{i=1}^N \frac{q_i}{A}\|f_i\|^2 = 1.$  If for any j, $1\leq j \leq N,$ \;$\frac{q_j}{A} \|f_j\|^2 < 1 ,$ then $n =  \sum_{i=1}^N  \langle f_i,\frac{1}{A}f_i \rangle < \sum_{i=1}^N \frac{1}{q_i} =n. $ Which is not possible. Therefore, $\frac{q_j}{A} \|f_j\|^2 = 1 , \;\;1 \leq i \leq N.$
\end{proof}\hfill{$\square$}

\section{Examples}
\begin{example}
 Let $ H= {\mathbb{C}}^2$ and consider a frame $F= \{f_1,f_2,f_3\},$ where $f_1 = \left[\begin{array}{l}
1 \\0
\end{array}\right] $,
$f_2 = \left[\begin{array}{l}
0 \\ 1/2
\end{array}\right] $
 and $ f_3 = \left[\begin{array}{l}
0 \\ 1/2
\end{array}\right], $ with probability sequence $\{p_i \}_{i=1}^3 = \left\{0, \frac{1}{2}, \frac{1}{2}\right\}.$ Therefore the weight number sequence is $\{q_i \}_{i=1}^3 = \left\{1, 2, 2 \right\}.$ \\
Then the canonical dual of $F$ is
$$S_{F}^{-1} F = \left\{ \left[\begin{array}{r} 1 \\ 0 \end{array}\right],  \left[\begin{array}{r} 0 \\ 1 \end{array}\right], \left[\begin{array}{l} 0 \\ 1  \end{array}\right] \;\right\}. $$\\
\noindent
It is easy to calculate that, $r_{P}^{(1)}(F,S_{F}^{-1}F) = max_{i=1}^3 q_i |\langle f_i,S_{F}^{-1} f_i \rangle | = 1. $ Therefore, the canonical dual is not only the 1-erasure PSOD-frame of $F,$ also the dual pair $(F,S_{F}^{-1} F)$ is an 1-erasure PSOD-pair.\\
Similarly, we can calculate that $\mathcal{O}_{P}^{(1)}(F,S_{F}^{-1}F) = max_{i=1}^3 q_i \| f_i\|\;\|S_{F}^{-1} f_i \| = 1. $Therefore, the canonical dual is not only the 1-erasure POD-frame of $F,$ also the dual pair $(F,S_{F}^{-1} F)$ is an 1-erasure POD-pair.\\
Accordingly, $\mathcal{A}_{P}^{(1)}(F,S_{F}^{-1}F) = max_{i=1}^3 q_i \| f_i\|\;\|S_{F}^{-1} f_i \| = 1. $Therefore, the canonical dual is not only the 1-erasure PASOD-frame of $F,$ also the dual pair $(F,S_{F}^{-1} F)$ is an 1-erasure PASOD-pair.\\

The set of duals of $F$ is of the form  
$$ D = \left\{ \left[\begin{array}{r} 1 \\ 0 \end{array}\right], \left[\begin{array}{r} \alpha \\ 1 - \beta \end{array}\right], \left[\begin{array}{l} -\alpha \\ 1-\beta \end{array}\right] \right\}; $$  where $\alpha, \beta \in \mathbb{C}.$\\

Let, $G= \{ g_i \}_{i=1}^3 = \left\{ \left[\begin{array}{r} 1 \\ 0 \end{array}\right], \left[\begin{array}{r} \alpha \\ 1 - \beta \end{array}\right], \left[\begin{array}{l} -\alpha \\ 1-\beta \end{array}\right] \right\}; $  where $\alpha, \beta \in \mathbb{C},$ be a dual of $F.$\\
Then, $r_{P}^{(1)}(F,G) = max_{i=1}^3 q_i |\langle f_i,g_i \rangle | = max_{i=1}^3 \bigg\{1, |1+ \beta |, | 1 - \beta| \bigg\}. $ Therefore, $G$ be an 1-erasure PSOD-frame of $F$ or $(F,G)$ dual pair is an 1-erasure PSOD-pair in $\mathcal{H}$ if and only if $\beta = 0.$ In other words, for any value of $\alpha \in \mathbb{C}$ and $\beta = 0,$ $G$ become 1-erasure PSOD-frame of $F$ and $(F,G)$ be an 1-erasure PASOD-pair in $\mathcal{H}.$ So, neither PSOD-frame of $F$ is unique nor PSOD-pair for $F$ is unique for 1-erasure.\\

\noindent
Similarly, $\mathcal{O}_{P}^{(1)}(F,G) = max_{i=1}^3 q_i \|f_i\|\;\|g_i\| = max_{i=1}^3 \bigg\{1, \left|\sqrt{\alpha^2 + (1+ \beta)^2}\right|, \left|\sqrt{\alpha^2 + (1- \beta)^2}\right|\bigg\}. $ If $\beta$ is any non-zero complex number then either $\left|\sqrt{\alpha^2 + (1+ \beta)^2}\right| > 1$ or $\left|\sqrt{\alpha^2 + (1- \beta)^2}\right| > 1,$ according to $Re (\beta) > 0 \;\text{or}\; Re (\beta) < 0.$ Therefore, $G$ is an 1-erasure PSOD-frame of $F$ or $(F,G)$ is an 1-erasure PSOD-pair in $\mathcal{H}$ implies $\beta = 0.$ For $\beta = 0,$ if $\alpha \neq 0$ then $\left|\sqrt{\alpha^2 + (1+ \beta)^2}\right| > 1.$ Therefore,  $G$ is an 1-erasure PSOD-frame of $F$ or $(F,G)$ is an 1-erasure PSOD-pair in $\mathcal{H}$ implies $\alpha = 0.$ Consequently,  $G$ is an 1-erasure PSOD-frame of $F$ or $(F,G)$ is an 1-erasure PSOD-pair in $\mathcal{H}$ if and only if $\alpha = \beta = 0.$ So, the canonical dual is the unique 1-erasure POD-frame of $F$ and $(F,S_{F}^{-1})$ be the unique 1-erasure POD-pair corresponding to $F$ in $\mathcal{H}.$\\

Also, it is easy to calculate that,
\begin{align*}
    \mathcal{A}_{P}^{(1)}(F,G) &= max_{i=1}^3  \frac{q_i \bigg\{|\langle f_i , g_i \rangle | + \| f_i\| \| g_i \| \bigg\}}{2} \\&= max_{i=1}^3 \bigg\{1,\frac{|1 + \beta| + \left|\sqrt{\alpha^2 + (1+ \beta)^2} \right|}{2}, \frac{|1 - \beta| + \left|\sqrt{\alpha^2 + (1 - \beta)^2} \right|}{2}\bigg\}.
\end{align*}

Using the similar argument as above we can conclude that $G$ is an 1-erasure PASOD-frame of $F$ or $(F,G)$ is an 1-erasure PASOD-pair in $\mathcal{H}$ if and only if $\alpha = \beta = 0.$  So, the canonical dual is the unique 1-erasure PASOD-frame of $F$ and $(F,S_{F}^{-1})$ be the unique 1-erasure PASOD-pair corresponding to $F$ in $\mathcal{H}.$
\end{example}

\begin{example}
 Let $ H= {\mathbb{C}}^2$ and consider a frame $F= \{f_1,f_2,f_3\},$ where $f_1 = \left[\begin{array}{l}
1 \\0
\end{array}\right] $,
$f_2 = \left[\begin{array}{l}
0 \\ 1
\end{array}\right] ,$
 $ f_3 = \left[\begin{array}{l}
\frac{1}{\sqrt{2}} \\ \frac{1}{\sqrt{2}}
\end{array}\right] ,$ with probability sequence $\{p_1,p_2,p_3\}, $ where $p_1 = \frac{1}{4}, \; p_2 = \frac{1}{4}, p_3 = \frac{1}{2}.$ Therefore, the weight number sequence is  $\{q_1,q_2,q_3\}, $ where $q_1 = \frac{4}{3}, \; q_2 = \frac{4}{3}, q_3 = 2.$ \\
Therefore, it is easy to compute that 
$S_F= \begin{bmatrix}
3/2 & 1/2\\
1/2 &  3/2 
\end{bmatrix} $ and the canonical dual is :
$$S_{F}^{-1} F = \left\{ \left[\begin{array}{r} 3/4 \\ -1/4 \end{array}\right],  \left[\begin{array}{r} -1/4 \\ 3/4 \end{array}\right], \left[\begin{array}{l} 1/{2\sqrt{2}} \\ 1/{2\sqrt{2}}  \end{array}\right] \right\} .$$ 
\noindent
Thus, $q_i \langle f_i , S_{F}^{-1}f_i \rangle = 1,\; i=1,2,3$ \; and $q_1 \|f_1\|\;\left\| S_{F}^{-1}f_1 \right\| = \frac{\sqrt{10}}{3} ,\;q_2 \|f_2\|\;\left\| S_{F}^{-1}f_2 \right\| = \frac{\sqrt{10}}{3},\;q_1 \|f_3\|\;\left\| S_{F}^{-1}f_3\right\| = 1.$ Therefore, $r_{P}^{(1)}( F, S_{F}^{-1}F) = 1$\; , $\mathcal{O}_{P}^{(1)}( F, S_{F}^{-1}F) = \frac{\sqrt{10}}{3} = 1.0540925 $ and  $\mathcal{A}_{P}^{(1)}( F, S_{F}^{-1}F) = \frac{\sqrt{10} + 3}{6} = 1.0270462767 .$ \\

 The set of duals of $F$ is of the form \\
$$D = \left\{ \left[\begin{array}{r} 3/4 + \alpha  \\ - 1/4 +\beta \end{array}\right], \left[\begin{array}{r} - 1/4 +\alpha  \\  3/4 +\beta   \end{array}\right], \left[\begin{array}{l}1/{2\sqrt{2}} - \sqrt{2}\alpha \\ 1/{2\sqrt{2}} - \sqrt{2}\beta \end{array}\right]  \right\}, \text{where}\; \alpha, \beta \in \mathbb{C}.$$
If we take $ \alpha = \beta = -.01,$ then the dual is \\
$$ G =\{g_i\}_{i=1}^3 = \left\{ \left[\begin{array}{r} 0.74  \\ -0.26 \end{array}\right], \left[\begin{array}{r} - 0.26 \\  0.74   \end{array}\right], \left[\begin{array}{l}{1.04}/{2\sqrt{2}}  \\ {1.04}/{2\sqrt{2}}  \end{array}\right]  \right\}.$$
Thus, $r_{P}^{(1)}( F,G) = max_{i=1}^3 q_i|\langle f_i , g_i \rangle |= 1.04,\; \mathcal{O}_{P}^{(1)}( F,G) = max_{i=1}^3 q_i\| f_i\|\;\| g_i \| = 1.045796 $ \; \text{and} $\mathcal{A}_{P}^{(1)}( F,G) = \max_{i=1}^3 \frac{q_i \bigg\{|\langle f_i , g_i \rangle | + \| f_i\| \| g_i \| \bigg\}}{2} = 1.0162313.$\\~\\
Therefore, the canonical dual is 1-erasure PSOD of $F$ and $(F,S_{F}^{-1} F)$ is an 1-erasure PSOD-pair in $\mathcal{H}.$ But, the canonical dual is neither POD nor PASOD for 1-erasure. Also,
 $(F,S_{F}^{-1} F)$  is neither POD-pair nor PSOD-pair for 1-erasure.\\
 As in Theorem(\ref{thm3point6}), $F_1 = span\{f_1,f_2\}$ and $F_2 = span\{f_3 \}.$ Therefore $F_1 \cap F_2 \neq \{0\}.$ Therefore canonical dual is not a POD for 1-erasure.\\
  As in Theorem(\ref{thm3point8}), $H_1 = span\{f_1,f_2\}$ and $H_2 = span\{f_3 \},$ and so $H_1 \cap H_2 \neq \{0\}$ and the canonical dual is not an 1-erasure PASOD.

\end{example}

\begin{example}
  
  Let $ H= {\mathbb{C}}^2$ and consider a frame $F= \{f_1,f_2,f_3\},$ where $f_1 = \left[\begin{array}{l}
1 \\0
\end{array}\right] $,
$f_2 = \left[\begin{array}{l}
0 \\ 1
\end{array}\right] $
 and $ f_3 = \left[\begin{array}{l}
1 \\ 1
\end{array}\right] $ with probability sequence $\{p_i \}_{i=1}^3 = \left\{\frac{1}{2}, \frac{1}{3}, \frac{1}{6} \right\}.$ Therefore the weight number sequence is $\{q_i \}_{i=1}^3 = \left\{2, \frac{3}{2}, \frac{6}{5} \right\}.$ \\
Then the canonical dual of $F$ is
$$S_{F}^{-1} F = \left\{ \frac{1}{3}\left[\begin{array}{r} 2 \\ -1 \end{array}\right],  \frac{1}{3}\left[\begin{array}{r} -1 \\ 2 \end{array}\right], \frac{1}{3}\left[\begin{array}{l} 1 \\ 1  \end{array}\right] \;\;\right\}. $$\\
\noindent
Thus, $r_{P}^{(1)}(F,S_{F}^{-1}F) = max_{i=1}^3 q_i \left|\langle f_i,S_{F}^{-1} f_i \rangle \right| = \frac{4}{3},\;\;  \mathcal{O}_{P}^{(1)}(F,S_{F}^{-1}F) = max_{i=1}^3 q_i \left\| f_i\right\|\;\left\|S_{F}^{-1} f_i \right\| = \frac{2\sqrt{5}}{3}$ and\;\;$\mathcal{A}_{F,S_{F}^{-1}F}^{(1)} = max_{i=1}^3 \dfrac{{q_i\bigg\{\left| \langle f_i,S_{F}^{-1}f_i \rangle \right| + \|f_i\|\; \left\|S_{F}^{-1}f_i\right\| \bigg\}}}{2} = \frac{2+\sqrt{5}}{3} . $\\
\noindent
Now,\; $ q_1\left\{\left\| S_{F}^{-1/2}f_1 \right\|^2 + \|f_1\|\;\left\| S_{F}^{-1}f_1\right\|\right\} = \frac{4 + 2\sqrt{5}}{3},\;q_2\left\{\left\| S_{F}^{-1/2}f_2 \right\|^2 + \|f_2\|\;\left\| S_{F}^{-1}f_2\right\|\right\}= \frac{2 + \sqrt{5}}{2},$ and\\\;$\left\{ \left\| S_{F}^{-1/2}f_3 \right\|^2 + \|f_3\|\;\left\| S_{F}^{-1}f_3\right\|\right\} = \frac{4}{5}$\\

\noindent
As theorem (\ref{thm3point8}) $ H_1 = span\left\{ \left[\begin{array}{l}
1 \\0
\end{array}\right] \right\},$   and $H_2 = span \left\{  \left[\begin{array}{l}
0 \\ 1
\end{array}\right], \left[\begin{array}{l}
1 \\ 1
\end{array}\right]\right\}.$ \\
Therefore, $H_1 \cap H_2 \neq \{0\}.$\\
The set of duals of $F$ is of the form  
$$G= \{ g_i \}_{i=1}^3 = \bigg\{ \left[\begin{array}{r} 2/3 + \gamma  \\ -1/3 + \delta \end{array}\right], \left[\begin{array}{r} -1/3 + \gamma \\ 2/3 + \delta \end{array}\right], \left[\begin{array}{l} 1/3 - \gamma \\ 1/3 -\delta  \end{array}\right] \bigg\}; $$  where $\gamma, \delta \in \mathbb{C}.$
\noindent
If we take $\gamma = \delta = -\frac{1}{6},$\; then the dual is $G' = \{ g_i \}_{i=1}^3 = \bigg\{ \left[\begin{array}{r} 1/2   \\ -1/2 \end{array}\right], \left[\begin{array}{r} -1/2  \\ 1/2 \end{array}\right], \left[\begin{array}{l} 1/2 \\ 1/2  \end{array}\right] \bigg\}.$
It can be easily calculate that $r_{P}^{(1)}(F,G') = \frac{6}{5}, \mathcal{O}_{P}^{(1)}(F,G') = \sqrt{2}$ and $\mathcal{A}_{P}^{(1)}(F,G') = \frac{1+ \sqrt{2}}{2}.$\\
So the canonical dual is neither PSOD-frame nor POD-frame nor PASOD-frame of $F$ for 1-erasure.\\~\\
\end{example}

\begin{example}
Let $ H= {\mathbb{R}}^2$ and consider a frame $F= \{f_1,f_2,f_3\},$ where $f_1 = \left[\begin{array}{l}
1 \\0
\end{array}\right] $,
$f_2 = \left[\begin{array}{l}
-1/2 \\ \sqrt{3}/2
\end{array}\right] $
 and $ f_3 = \left[\begin{array}{l}
-1/2 \\ -\sqrt{3}/{2}
\end{array}\right] $ with probability sequence $\{p_i \}_{i=1}^3 = \left\{\frac{1}{3}, \frac{1}{3}, \frac{1}{3} \right\}.$ Therefore the weight number sequence is $\{q_i \}_{i=1}^3 = \left\{\frac{3}{2}, \frac{3}{2}, \frac{3}{2}\right\}.$ \\
So the canonical dual is self dual.\\
Therefore, $r_{P}^{(1)}(F,F) = \mathcal{O}_{P}^{(1)}( F,G)= \mathcal{A}_{P}^{(1)}( F,G)= max_{i=1}^3 q_i \|f_i\|^2 = \frac{3}{2}.$\\
In \cite{jins}, example 2.3, it is proved that the canonical dual is the unique optimal dual for 1-erasure. Therefore, the canonical dual is unique POD for 1-erasure. But $(F,S_{F}^{-1}F)$ is not a POD-pair for 1-erasure.

 The set of duals of $F$ is of the form \\
$$D = \left\{ \left[\begin{array}{r} 1 + \alpha  \\ \beta \end{array}\right], \left[\begin{array}{r} - 1/2 +\alpha  \\  \sqrt{3}/{2} +\beta   \end{array}\right], \left[\begin{array}{l}\ -1/2 +\alpha  \\ -\sqrt{3}/{2} +\beta \end{array}\right]  \right\}, \text{where}\; \alpha, \beta \in \mathbb{R}.$$
Let  $G= \left\{ \left[\begin{array}{r} 1 + \alpha  \\ \beta \end{array}\right], \left[\begin{array}{r} - 1/2 +\alpha  \\  \sqrt{3}/{2} +\beta   \end{array}\right], \left[\begin{array}{l}\ -1/2 +\alpha  \\ -\sqrt{3}/{2} +\beta \end{array}\right]   \right\},$ be a dual of $F.$\\
Then, 
$$r_{P}^{(1)}(F,G) = max \left\{ \frac{3}{2} \left|1+ \alpha \right|, \frac{3}{2} \left| \frac{1}{2}\left(\frac{1}{2} - \alpha \right) + \frac{\sqrt{3}}{2} \left(\frac{\sqrt{3}}{2} + \beta \right) \right|, \frac{3}{2} \left| \frac{1}{2} \left(\frac{1}{2} - \alpha \right) + \frac{\sqrt{3}}{2}\left(\frac{\sqrt{3}}{2} - \beta \right) \right|  \right\} $$
\noindent
If $\alpha > 0,$ then $\frac{3}{2} \left|1+ \alpha \right| > \frac{3}{2} .$\\
If $\alpha < 0 \;and \;\beta >0,$ then $\frac{3}{2} \left| \frac{1}{2}\left(\frac{1}{2} - \alpha \right) + \frac{\sqrt{3}}{2} \left(\frac{\sqrt{3}}{2} + \beta \right) \right| > \frac{3}{2}$\\
If $\alpha < 0 \;and \;\beta < 0,$ then $\frac{3}{2} \left| \frac{1}{2} \left(\frac{1}{2} - \alpha \right) + \frac{\sqrt{3}}{2}\left(\frac{\sqrt{3}}{2} - \beta \right) \right| > \frac{3}{2}$\\
Therefore, $r_{P}^{(1)}(F,G) \geq r_{P}^{(1)}(F,S_{F}^{-1}F),$ for any dual $G$ of $F$. Hence the canonical dual is the unique PSOD-frame of $F$ for 1-erasure. But, $(F,S_{F}^{-1}F)$ is not a PSOD-pair for 1-erasure.
\end{example}

\section*{Acknowledgement}
The author is thankful to Prof. Devaraj P. for his guidance and help. The author is also thankful to Ankush Kumar Garg and Tathagata Sarkar  for reading the manuscript carefully and giving helpful inputs. The author is grateful to IISER-Thiruvananthapuram for providing fellowship.


\begin{thebibliography}{plain}
\bibitem{ole}Ole Christensen, An Introduction to Frames and Riesz Bases, Birkhauser, 2015.
\bibitem{jerr} Lopez, J. and Han, D., 2010. Optimal dual frames for erasures. Linear algebra and its applications, 432(1), pp.471-482.
\bibitem{miao}Miao, H., Leng, J., Yu, J. and Li, D., 2018. Probability modeled optimal K-frame for erasures. IEEE Access, 6, pp.54507-54515.
\bibitem{leng} Leng, J., Han, D. and Huang, T., 2013. Probability modelled optimal frames for erasures. Linear Algebra and its Applications, 438(11), pp.4222-4236.
\bibitem{casa} P. G. Casazza and J. Kovacevic, Uniform tight frames with erasures, Advances in Computational Mathematics Vol. 18, Nos. 2-4 (2003) pp.387-430.
\bibitem{goya} Goyal, V.K., Kovačević, J. and Kelner, J.A., 2001. Quantized frame expansions with erasures. Applied and Computational Harmonic Analysis, 10(3), pp.203-233.
\bibitem{jins}Leng, J. and Han, D., 2011. Optimal dual frames for erasures II. Linear algebra and its applications, 435(6), pp.1464-1472.
\bibitem{sali} Pehlivan, S., Han, D. and Mohapatra, R., 2013. Linearly connected sequences and spectrally optimal dual frames for erasures. Journal of Functional Analysis, 265(11), pp.2855-2876.
\bibitem{leng3}Leng, J., Han, D. and Huang, T., 2011. Optimal dual frames for communication coding with probabilistic erasures. IEEE transactions on signal processing, 59(11), pp.5380-5389.
\bibitem{bodm1} Bodmann, B.G. and Paulsen, V.I., 2005. Frames, graphs and erasures. Linear algebra and its applications, 404, pp.118-146.
\bibitem{casa2}Casazza, P.G. and Kovačević, J., 2003. Equal-norm tight frames with erasures. Advances in Computational Mathematics, 18(2), pp.387-430.
\bibitem{holm}Holmes, R.B. and Paulsen, V.I., 2004. Optimal frames for erasures. Linear Algebra and its Applications, 377, pp.31-51.
\bibitem{deep} Deepshikha and Samanta, A., 2022. Averaged numerically optimal dual frames for erasures. Linear and Multilinear Algebra, pp.1-16.
\bibitem{dana} Dana, A.F., Gowaikar, R., Palanki, R., Hassibi, B. and Effros, M., 2006. Capacity of wireless erasure networks. IEEE Transactions on Information Theory, 52(3), pp.789-804.
\bibitem{bodm2}Bodmann, B.G. and Paulsen, V.I., 2007. Frame paths and error bounds for sigma–delta quantization. Applied and Computational Harmonic Analysis, 22(2), pp.176-197.
\bibitem{cand}Candès, E.J. and Donoho, D.L., 2004. New tight frames of curvelets and optimal representations of objects with piecewise C2 singularities. Communications on Pure and Applied Mathematics: A Journal Issued by the Courant Institute of Mathematical Sciences, 57(2), pp.219-266.
\bibitem{alba}Albanese, A., Blomer, J., Edmonds, J., Luby, M. and Sudan, M., 1996. Priority encoding transmission. IEEE transactions on information theory, 42(6), pp.1737-1744.
\bibitem{duff}Duffin, R.J. and Schaeffer, A.C., 1952. A class of nonharmonic Fourier series. Transactions of the American Mathematical Society, 72(2), pp.341-366.
\bibitem{benn}Bennett, C.H., DiVincenzo, D.P. and Smolin, J.A., 1997. Capacities of quantum erasure channels. Physical Review Letters, 78(16), p.3217.
\bibitem{cido}Cidon, I., Kodesh, H. and Sidi, M., 1988. Erasure, capture, and random power level selection in multiple-access systems. IEEE Transactions on Communications, 36(3), pp.263-271.
\bibitem{cass2} Casazza, P.G. and Leon, M., 2010. Existence and construction of finite frames with a given frame operator. Int. J. Pure Appl. Math, 63(2), pp.149-157.
\end{thebibliography}
\end{document}